\documentclass[11pt]{article}
\usepackage{amssymb}
\usepackage{parskip}
\usepackage{setspace}
\usepackage{graphicx}
\usepackage[]{epsfig}
\usepackage{amsmath}
\usepackage{amssymb}
\usepackage{stmaryrd}
\usepackage{verbatim}
\usepackage{amsthm}
\usepackage{url}
\usepackage{latexsym}
\usepackage{easybmat}
\usepackage{amsmath}
\usepackage{hyperref}
\usepackage{hypcap}
\hypersetup{colorlinks = true,linkcolor = blue, anchorcolor = blue, citecolor = blue, filecolor = red,urlcolor = blue}
\usepackage{color}
\numberwithin{equation}{subsection}
\numberwithin{equation}{section}
\begin{subequations}
\end{subequations}

\begin{document}
\numberwithin{equation}{subsection}
\numberwithin{equation}{section}
\begin{subequations}
\end{subequations}
\newtheorem{theorem}{Theorem}[section]
\newtheorem{lemma}{Lemma}
%\newproblem{problem}{Problem}
\newtheorem{definition}{Definition}
\newtheorem{proposition}[theorem]{Proposition}
\newtheorem{corollary}[theorem]{Corollary}

\date{}
%\topmargin +.25in
%\textheight 8in
%\textwidth 6in
%\parindent 0pt
%\parskip 18pt
%\baselineskip 12pt
\date{}
\title{\bf Dynamical behaviour of an ecological system with Beddington-DeAngelis functional response}
\author{Sahabuddin Sarwardi$^{(\dagger)}$\footnote{Author to whom all correspondence should be addressed}~, Md. Reduanur Mandal$^{(\dagger)}$ and Nurul Huda Gazi$^{(\dagger)} $  \\
$^{(\dagger)}$  Department of Mathematics, Aliah University, IIA/27, New Town\\
Kolkata - 700 156, West Bengal, India.\\
email: \href{s.sarwardi@gmail.com}{s.sarwardi@gmail.com}}

\maketitle

\begin{abstract}
The objective of this paper is to study the dynamical behaviour systematically of an ecological system with Beddington-DeAngelis functional response which avoids the criticism occurred in the case of ratio-dependent functional response at the low population density of both the species. The essential mathematical features of the present model have been analyzed thoroughly in terms of the local and the global stability and the bifurcations
arising in some selected situations as well. The threshold values for some parameters indicating
the feasibility and the stability conditions of some equilibria are also determined. We show that the dynamics outcome of the interaction among the species are much sensitive to the system parameters and initial population volume. The ranges of the
significant parameters under which the system admits a Hopf bifurcation are investigated. The explicit
formulae for determining the stability, direction and other properties of bifurcating periodic solutions
are also derived with the use of both the normal form and the central manifold theory (cf. Carr \cite{Carr}). Numerical illustrations
are performed finally in order to validate the applicability of the model under consideration.\end{abstract}



\section{Introduction}
Mathematical model is an important tool in analyzing the ecological models. Ecological problems are challenging and important issues from  both the ecological and the mathematical point of view (cf. Anderson and May \cite{May81}, Beretta and Kuang \cite{Beretta98}, Freedman \cite{F90}, Hadeler and Freedman \cite{HF89},  Hethcote et al. \cite{HWM04}, Ma and Takeuchi \cite{Ma98}, Venturino \cite{V95}, Xiao and Chen \cite{XiaCh01}). The dynamic relationship between predator and its prey has long been and will continue to be one of the dominant themes in both ecology and mathematical ecology due to
its universal existence and importance. The most common method of modelling that ecological interactions consists of two differential equations with simple correspondence between the consumption of prey by the admissible predator and their population growth. The traditional predator-prey models have been
studied extensively (cf. Cantrell and Cosner \cite{Cantrell}, Cosner et al. \cite{Cosner99}, Cui and Takeuchi \cite{Cui}, Huo et al. \cite{Huo} and Hwang \cite{Hwang}), but those are questioned by
several biologists. The most crucial element in these models is the ``functional response"-- the expression that describes the rate at which the number of prey consumed by a predator. Modifications were limited to replacing the Malthusian growth function, the predator per capita consumption of prey functions such as Holling type I, II, III functional responses or density dependent mortality rates. These functional responses depend only on the prey volume $x$, but soon it became clear that the predator volume $y$ can influence this function by direct interference while searching or by pseudo interference (cf.  Curds and Cockburn \cite{Curds}, Hassell and Varley \cite{Hassell} and Salt \cite{Salt}). A simple way of incorporating predator dependence in the functional response was proposed by Arditi and Ginzburg \cite{Ginzburg}, who considered this response function as a function of the ratio $x/y$. The ratio-dependent response function produces richer dynamics than that of all the Holling types responses, but it is often criticized that the paradox occurred at the low densities of both populations size. Normally one would expect that the population growth rate decrease when both the populations fall bellow some critical volume, because food-searching effort becomes very high. For some ecological interaction ratio-dependent model give the negative feed back. Thus, the Lotka-Volterra type predator-prey model with the Beddington-DeAngelis functional response has been proposed and well studied. Keeping these in mind, the proposed model can
be expressed as follows:
\begin{eqnarray}
\left \{ \begin{array}{ll}
\frac{dx'_1}{dt} = rx'_1\bigl(1-
\frac{x'_1}{k}\bigr)-\frac{c_{1}x'_1 x'_2}{a_{1}+x'_1+b_{1}x'_2}\\
\frac{dx'_2}{dt} = -\delta_{1}x'_2+\frac{c_{1}e_{1}x'_1x'_2}{a_{1}+x'_1+b_{1}x'_2}\end{array}\right.
 \label{eq1}\end{eqnarray}

 with the initial conditions ${x'_1}(0) = {x'_{1}}^{0} > 0$ and ${x'_{2}}(0) ={x'_{2}}^{0} >0$. The functions ${x'_1}(t), ~{x'_2}(t)$ are the volumes of prey and predator at any time $t.$  All the system parameters are assumed to be positive and have their usual biological meanings.
The functional response $\frac{c_{1}x'_1 x'_2}{a_{1}+x'_1+b_{1}x'_2}$ in system (\ref{eq1}) was introduced by Beddington \cite{Beddington75} and DeAngelis et al. \cite{DeAngeli75} as a solution of the observed problem in the classic predator-prey theory. It is similar to the well-known Holling type-II functional response but has an
extra term $b_1 x_2$ in the denominator which models mutual interference between predators. It represents the most qualitative features of the ratio-dependent models, but avoids the ``low-densities problem", which usually the source of controversy.
It can be derived mechanistically from considerations of time utilization (cf. Beddington \cite{Beddington75}) or spatial
limits on predation.

The present study under consideration has been carried out sequentially in the latter sections as follows: The basic assumptions and the model formation are proposed in Section \ref{model}. Section \ref{preli} deals with some preliminary results. The equilibria and their feasibility are rightly given in Section \ref{equilibria}. The local analyses of the system around the boundary as well as interior equilibria are discussed in Section \ref{localstability}. The global analysis of the system around the interior equilibrium is studied at length in Section \ref{Direction}. Simulation results are reported in Section \ref{Numerical} while a final discussion and interpretation of the results of the present study in ecological terms are rightly included in the concluding Section \ref{Discussion}.

\section{Model formulation}\label{model}
Firstly we replaced the logistics growth function $r x_1(1-\frac{x_1}{k})$ of the prey species by the modified quasi-linear growth function $rx_1(1-\frac{x_1}{x_1+k}) = r (\frac{k}{k+x_1}) x_1 = r' x_1$ $(r'<r)$ in order to make the model free from any axial equilibrium. Which fits better for some special type of ecosystem, whereof environmental carrying capacity varies w.r.t. its prey volume, i.e., carrying capacity is always greater than its present prey volume. In the present model we introduce one more predator species in the model (\ref{eq1}) to make it one step closure to reality. Thus, our final model is extended to the following form:

\begin{eqnarray}
\left \{ \begin{array}{lll}
  \frac{dx_1}{dt} = rx_1(1-\frac{x_1}{x_1+k})-\frac{c_{1}x_1 x_2}{a_{1}+x_1+b_{1}x_2}
                  -\frac{c_{2}x_1x_3}{a_{2}+x_1+b_{2}x_3} \\
  \frac{dx_2}{dt} = -\delta_{1}x_2+\frac{c_{1}e_{1}x_1x_2}{a_{1}+x_1+b_{1}x_2},\\
  \frac{dx_3}{dt} = -\delta_{2}x_3+\frac{c_{2}e_{2}x_1x_3}{a_{2}+x_1+b_{2}x_3}
\end{array}\right.
 \label{eq2}\end{eqnarray} where
 $x_1$ is the population volume of the two prey species and $x_2,$ $x_3$ are the population
 volumes of the predator species at any time $t$. It is assumed that all the system parameters are positive constants. Here $r$ and $k$ are the growth rate and the half-saturation constant for the prey
species, $\delta_1,$ $\delta_2$ are the first and second predator's death rate respectively.  $c_{1}$, $c_{2}$ are the respective search rates of the first and second predator on the prey species,~ $\frac{c_1}{a_1},$
 $\frac{c_2}{a_2}$ are the maximum number of prey that can be
eaten by the first and second predator per unit time respectively; $\frac{1}{a_1}$, $\frac{1}{a_2}$ being their respective half saturation rates while $e_1$, $e_2$ are the conversion factors, denoting the number of newly born first and second predator
for each captured prey species respectively $(0<e_1,~e_2<1)$. The parameters $b_1$ and $b_2$ measure the coefficients of mutual interference among the first and second predator species respectively. The terms $\frac{c_{1}x_1 x_2}{a_{1}+x_1+b_{1}x_2}$
and $\frac{c_{2}x_1x_3}{a_{2}+x_1+b_{2}x_3}$ denote the respective predator responses on the first and second prey species.
This type of predator response function is known as Beddington-DeAngelis response function (cf.
Beddington \cite{Beddington75} and DeAngelis et al. \cite{DeAngeli75}).

\section{Some preliminary results}\label{preli}
\subsection{Existence and positive invariance}
Letting, $x = (x_1,x_2,x_3)^t,$ $f : \mathbf{R}^3 \rightarrow \mathbf{R}^3,$ $F = (f_1,f_2,f_3)^t,$
the system (\ref{eq2}) can be rewritten as $\dot{x} = f(x)$. Here
$f_i \in C^{\infty}(\mathbf{R})$~for $i=1,2,3,$~ where
$f_1=rx_1(1-\frac{x_1}{x_1+k})-\frac{c_{1}x_1 x_2}{a_{1}+x_1+b_{1}x_2}
                  -\frac{c_{2}x_1x_3}{a_{2}+x_1+b_{2}x_3}$, $f_2=-\delta_{1}x_2+\frac{c_{1}e_{1}x_1x_2}{a_{1}+x_1+b_{1}x_2}$ and $f_3=-\delta_{2}x_3+\frac{c_{2}e_{2}x_1x_3}{a_{2}+x_1+b_{2}x_3}.$
Since the vector function $f$ is a smooth function of the variables $(x_1,x_2,x_3)$ in the positive octant $\Omega^0=\{(x_1,x_2,x_3):~x_1>0,x_2>0,x_3>0\},$ the local existence and uniqueness of the solution of the system (\ref{eq2})  hold.

\subsection{Persistence}\label{persistence}
If a compact set $D\subset \Omega^0=\{(x_1,x_2,x_3):~ x_{i}> 0, ~ i =1, 2, 3\}$ exists such that all solutions of (\ref{eq2})
eventually enter and remain in $D$, the system is called persistent.

\begin{proposition}
The system (\ref{eq2}) is persistent if the conditions:
$(i)\, r > \delta_1 + \delta_{2},~
(ii) ~ x_{1_{1}} >\frac{a_{2}\delta_{2}}{c_{2} e_{2} -\delta_{2}}, \, (iii) \, x_{1_{2}} >\frac{a_{1}\delta_{1}}{c_{1} e_{1} -\delta_{1}}$ are satisfied.
\end{proposition}

\begin{proof}
We use the method of average Lyapunov function (cf. Gard and Halam \cite{GaHa}), considering a function of the form
$$
V(x_1,x_2,x_3) = x_1^{\gamma_{1}} x_2^{\gamma_{2}} x_3^{\gamma_{3}},
$$
where $\gamma_{1},  \gamma_{2}$ and $\gamma_{3}$ are positive constants to be determined. We define
\begin{eqnarray}
\Pi(x_1,x_2,x_3) &=&\frac{\dot V}{V}\nonumber\\
&=& \gamma_{1}\Bigl( r-\frac{r x_1}{x_1+k}-\frac{c_{1} x_2}{a_{1}+x_1+b_{1}x_2}
                  -\frac{c_{2}x_3}{a_{2}+x_1+b_{2}x_3}\Bigr) + \gamma_{2}\Bigl(-\delta_{1}+\frac{c_{1}e_{1}x_1}{a_{1}+x_1+b_{1}x_2}\Bigr)  \nonumber\\ &~&+ \gamma_{3}\Bigl(-\delta_{2}+\frac{c_{2}e_{2}x_1}{a_{2}+x_1+b_{2}x_3}\Bigr).\nonumber\end{eqnarray}
We now prove that this function is positive at each boundary equilibrium. Let $\gamma_{i} =\gamma,$ for $i= 1,~ 2,~3.$ In fact at $E_{0},$ we have
$\Pi(0,0,0) =  r - \delta_1 - \delta_{2} > 0$  from the condition (i).
Moreover, from condition ($ii$) and $(iii)$, we find the values of $\Pi$ at $E_{1}$ and $E_{2}$ respectively,
\begin{eqnarray*}
\Pi(x_{{1}_1},x_{{2}_1},0) = \gamma\Bigl(-\delta_{2}+\frac{e_{2}c_{2}x_{1_{1}}}{a_{2}+ x_{{1}_1}}\Bigr) > 0,\\
\Pi(x_{{1}_2},0,x_{{3}_2}) = \gamma\Bigl(-\delta_{1}+\frac{e_{1}c_{1}x_{1_{2}}}{a_{1}+ x_{{1}_2}}\Bigr) > 0 .
\end{eqnarray*}

Hence, there always exists a positive number $\gamma$ such that $\Pi > 0$ at the boundary equilibria.
Hence $V$ is an average Lyapunov function and
thus, the system (\ref{eq2}) is persistent.
\end{proof}

Since the system is uniformly persistent, there exists $\sigma > 0$ and $\tau > 0$ such
that $x_{i}(t) > \sigma,$~for all~ $t > \tau,$ $i = 1,~ 2,~ 3.$

\subsection{Boundedness}\label{bounded}
Boundedness implies that the system is consistent with biological significance.
The following propositions ensure the boundedness of the system (\ref{eq2}).
\begin{proposition}\label{Prop1}
The prey population is always bounded from above.\label{prop3.1}
\end{proposition}
\begin{proof}
Before proving that the prey population is bounded above, we need to prove that the predator populations $x_{2}$ and $x_3$ are bounded above. To prove this result, considering the second sub equation of the system  (\ref{eq2}) and one can obtain the following differential inequality:
\begin{eqnarray}
\frac{dx_2}{dt}\leq- (\delta_1-c_1 e_1)x_{2}.\nonumber
\end{eqnarray}
Integrating the above differential inequality between the limits 0 and $t,$ we have
$x_{2}(t)\leq x_{2}(0) e^{-(\delta_1-c_1 e_1)t}.$ Thus, if $(\delta_1-c_1 e_1)>0,$ then it is obviously found a positive number $\tau_{1}$ there exists a positive constant $m_1$ such that $x_{2}(t) \leq m_1,$ for all $t\geq \tau_{1}.$ By using the similar argument, one can obtain that, if $(\delta_2-c_2 e_2)>0,$ then corresponding to a positive number $\tau_{2}$ there exists a positive constant $m_2$ such that $x_{3}(t) \leq m_2,$ for all $t\geq \tau_{2}.$ Both the results can be written unitedly as $x_{i}> m = \min{(m_{1}, m_{2})}$ for all $t> \tau_3= \max{(\tau_1, \tau_2)},$ $i =2,3,$ with the additional condition $\min{(\delta_1-c_1 e_1,~\delta_2-c_2 e_2)} >0.$

Now from the first sub-equation of (\ref{eq2}), the following inequality is found
\begin{eqnarray}
\frac{dx_1}{dt}\leq  \frac{\bigl((e_1+e_2)\sigma -rk\bigr)x_1}{km}\Bigl(\frac{k \bigl(r m-(e_1+e_2)\sigma\bigr)}{(e_1+e_2)\sigma -rk}-x_1\Bigr).\nonumber
\end{eqnarray}
Hence, by using standard but simple argument, we have
\begin{eqnarray}
\limsup_{t\rightarrow +\infty}x_1(t)\leq  \frac{k r m -(e_1+e_2) k\sigma}{(e_1+e_2)\sigma -rk}= w, \, \mbox{where}\,\, \frac{rm}{e_1+e_2}<\sigma<\frac{rk}{e_1+e_2}.\nonumber
\end{eqnarray}
\end{proof}
\begin{proposition}  The solutions of (\ref{eq2}) starting
in $\Omega^0$ are uniformly bounded with an ultimate bound.\label{prop2}
\end{proposition}
 {\textbf{Proof.}}
Considering the total environment population $\chi=x_1+\frac{x_2}{e_1}+\frac{x_3}{e_2}.$ Using the theorem on differential inequality (cf. Birkhoff and Rota \cite{birkhoff1959ordinary}) and following the steps of Haque and Venturino \cite{HV}, Sarwardi et al \cite{sarwardijbp13}, boundedness of the solution trajectories of
this model is established. In particular,
\begin{eqnarray}\label{xub}
\limsup_{t\rightarrow +\infty}{\Bigl(x_1+\frac{x_2}{e_1}+\frac{x_3}{e_2}\Bigr)} \leq \frac{(r+1)k + w}{\rho}= M, \, \mbox{where}~ \rho =\min{(1, \delta_1, \delta_2)},
\end{eqnarray}
with the last bound is independent of the initial condition.

Hence, all the solutions of (\ref{eq2}) starting in $\mathbf{R}_{+}^3 $ for any $\theta> 0$ evolve with respect to time in the compact region
\begin{eqnarray}
\bar{\Omega} =\left\{(x_1,x_2,x_3)\in \mathbf{R}_{+}^3 : x_1+\frac{x_2}{e_1}+\frac{x_3}{e_2}\le M + \theta \right\}.\label{eqbdd1}
\end{eqnarray}
\section{Equilibria and their feasibility}\label{equilibria}
The equilibria of the dynamical system (\ref{eq2}) are given as follows:

     1. The trivial equilibrium point $E_{0}(0,0,0)$ is always feasible.

     2. (a) The first boundary equilibrium point is $E_{1}(x_{1_1},x_{2_1},0).$ The component
     $x_{1_1}$ is a root of the quadratic equation $l_1 x^2_{1_1}+(l_2+l_1 k+r k b_1 e_1) x_{1_1}+l_2 k=0,$
     where $l_1=(\delta_1-c_1e_1)$, $l_2=a_1\delta_1.$ If $l_1<0,$ then the quadratic equation in $x_{1_1}$ possesses a unique positive root and consequently  $x_{2_1}=\frac{(c_{1}e_{1}-\delta_{1})x_{1_1}-a_{1}\delta_{1}}{b_{1}\delta_{1}}.$ The feasibility of the equilibrium $E_1$ is maintained if the condition $x_{1_1}> \frac{b_1 \delta_1}{c_{1}e_{1}-\delta_{1}}$ is satisfied.

2. (b) The second boundary equilibrium point is $E_{2}(x_{1_2},0,x_{3_2}).$ The component $x_{1_2}$ is the root of the quadratic equation
    $m_1 x^2_{1_2}+(m_2+m_1 k+r kb_2 e_2) x_{1_2}+m_2 k=0,$
     where $m_1=(\delta_2-c_2e_2)$, $l_2=a_2\delta_2.$ If $m_1<0,$ then the quadratic equation in $x_{1_2}$ possesses a unique positive root and consequently $x_{3_2}=\frac{(c_{2}e_{2}-\delta_{2})x_{1_2}-a_{2}\delta_{2}}{b_{2}\delta_{2}}.$ The feasibility of the equilibrium $E_2$ is maintained if the condition $x_{1_2}> \frac{b_2 \delta_2}{c_{2}e_{2}-\delta_{2}}$ holds.

3. The interior equilibrium point is $E_*(x_{1*}, x_{2*}, x_{3*}),$
where the first component $x_{1*}$ is the root of the following quadratic equation: \begin{eqnarray}n_1
x^2_{1*}+(n_2+ n_1 k+ rkb_{{1}}b_{{2}}e_{{1}}e_{{2}})x_{1*}+n_2 k=0,\label{inteq}\end{eqnarray} where
$n_1 = b_{{1}}e_{{1}}(\delta_{{2}}-c_2 e_2) +b_{{2}}e_2(\delta_{{1}}-c_{{1}}e_{{1}})$ and $n_2 =b_{{2}}e_{{2}}\delta_{{1}}a_{{1}}+b_{{1}}e_{{
1}}\delta_{{2}}a_{{2}}.$

Case I:  Let $n_1<0.$ In this case there exists exactly one
positive root of the quadratic equation (\ref{inteq}) irrespective of the sign of $(n_2+ n_1 k+ rkb_{{1}}b_{{2}}e_{{1}}e_{{2}}).$

Case II: Let $n_1>0$. In this case there are two possibilities: (i) if $n_2+ n_1 k+ rkb_{{1}}b_{{2}}e_{{1}}e_{{2}}>0$, then there is no positive
solution and (ii) if $n_2+ n_1 k+ rkb_{{1}}b_{{2}}e_{{1}}e_{{2}}<0$, then there exists two positive roots or no positive root.

In this present analysis we consider the Case I. Under this assumption the next two components of the interior equilibrium can be obtained as
$x_{2*}=\frac{(c_{1}e_{1}-\delta_{1})x_{1*}-a_{1}\delta_{1}}{b_{1}\delta_{1}},$~
$x_{3*}=\frac{(c_{2}e_{2}-\delta_{2})x_{1*}-a_{2}\delta_{2}}{b_{2}\delta_{2}}.$ The feasibility of this important equilibrium point $E_{*}$ is confirmed under the condition $x_{1*}>\max
\left\{\frac{a_1\delta_1}{c_1e_1-\delta_1},
\frac{a_2\delta_2}{c_2e_2-\delta_2} \right\}.$ Moreover, the positivity condition of second and third components of the interior equilibrium ensures the impossibility of the Case II.

\text{\bf Remark:} The feasibility and existences conditions of both the planer equilibria $E_{1}$ and $E_{2}$ immediately implies the existence of the unique feasible interior equilibrium point $E_{*}.$ But the existence of the unique feasible interior equilibrium point $E_{*}$ implies three possibilities: (i) $E_1$ exists and $E_2$ does not exist, (ii) $E_2$ exists and $E_1$ does not exist, (iii) existence of both.

\section{Local stability and bifurcation}\label{localstability}
 The Jacobian matrix $J(x)$ of the system (\ref{eq2}) at any point $ x=(x_1, x_2,x_3)$ is given by
  \begin{equation}\label{3.1}
 J(x)_{3\times3}=\left(
  \begin{matrix}
    \frac{r k^2}{(x_1+k)^{2}}-\frac{c_{1} x_2(a_{1}+b_{1}x_2)}{(a_1+x_1+b_{1} x_2)^2}-\frac{c_{2}x_3(a_2+b_{2}x_3)}{(a_{2}+x_1+b_{2}x_3)^{2}} &
    -\frac{c_{1} x_1 (a_{1}+x_1)}{(a_{1}+x_1+b_{1} x_2)^2} & -\frac{c_{2} x_1 (a_{2}+x_1)}{(a_2+x_1+b_{2} x_3)^2} \\
    \frac{c_{1}e_{1}x_2(a_{1}+b_{1} x_2)}{(a_{1}+x_1+b_{1} x_2)^2} & -\delta_1 + \frac{c_{1}e_{1}x_1(a_1+x_1)}{(a_{1}+x_1+b_{1} x_2)^2} & 0 \\
    \frac{c_{2}e_{2} x_3(a_2+b_{2} x_3)}{(a_{2}+x_1+b_{2} x_3)^2} & 0 & -\delta_2 + \frac{c_{2}e_{2} x_1(a_2+ x_1)}{(a_{2}+x_1+b_{2} x_3)^2}
  \end{matrix}
  \right).
  \end{equation}
Its characteristic equation is $\Delta(\lambda) = \lambda^3+k_{1}\lambda^2+k_{2}\lambda+k_{3}=0$,
where $k_{1} = -\text{tr}(J)$, $k_{2} = M$ and $k_{3} = -\det(J)$; $M$ being the sum
of the principal minors of order two of $J.$

Note that the conditions for occurrence of Hopf bifurcation are that there exists a
certain bifurcation parameter $r =r _{c}$ such that
$C_{2}(r _{c})=k _{1}(r _{c})k_{2}(r _{c})-k_{3}(r _{c})=0$ with $k_{2}(r_{c})> 0$
and $\frac{d}{dr}{(\mbox{Re}(\lambda(r)))}|_{r =r_{c}} \neq 0,$
where $\lambda$ is root of the characteristic equation $\Delta(\lambda)=0$.

  \subsection{Local analysis of the system around $E_0,~ E_{1},~ E_2$}
  {\bf Stability:}
 The eigenvalues of the Jacobian matrix $J(E_0)$ are $r, -\delta_1$ and $-\delta_2$. Hence $E_{0}$ is unstable in nature (saddle point). Let $J(E_1) =(\xi_{ij})_{3\times 3}$ and $J(E_2) =(\eta_{ij})_{3\times 3}.$ Using the Routh–Hurwitz
criterion, it can be easily shown that the eigenvalues of the matrices $J(E_1)$ and $J(E_2)$ will have negative real parts iff the conditions $e_1 x_{1_1} + x_{2_1}>\frac{k(1-b_1 e_1)-a_1}{b_1}$ and  $e_2 x_{1_2} + x_{3_2}>\frac{k(1-b_2 e_2)-a_2}{b_2}$ respectively.  Hence the equilibria $E_{1}$ and $E_2$ are locally asymptotically stable under the conditions $e_1 x_{1_1} + x_{2_1}>\frac{k(1-b_1 e_1)-a_1}{b_1}$ and $e_2 x_{1_2} + x_{3_2}>\frac{k(1-b_2 e_2)-a_2}{b_2}$ respectively (cf. Section 4 of of Sarwardi et al. \cite{Sarwardirefuge}).

  {\bf Bifurcation:}
Since the equilibrium point $E_0$ is a saddle in nature, hence, there is no question of Hopf bifurcation around this equilibrium. In order to have Hopf bifurcation around the equilibria $E_1$, $E_2$, it is sufficient to show that the coefficient of $\lambda$ in the quadratic factor of the characteristic polynomial of $J(E_k)$ $(k= 1, 2)$ is zero and the constant term is positive. The conditions for which annihilation of the linear terms in the quadratic factors of the characteristic polynomials of $J(E_1)$ and $J(E_2)$ can be made possible are ~$\xi_{11}+\xi_{22}=0$ and $\eta_{11}+\eta_{33}=0.$ For detailed analysis, interested readers are referred to Appendix A of Haque and Venturion \cite{HV06}. The parametric regions where Hopf bifurcations occur around $E_1$ and $E_2$ are respectively established by the equality constraints  $e_1 x_{1_1} + x_{2_1} = \frac{k(1-b_1 e_1)-a_1}{b_1}$ and  $e_2 x_{1_2} + x_{3_2}=\frac{k(1-b_2 e_2)-a_2}{b_2}$.

  \subsection{Local analysis of the system around the interior equilibrium}
  \begin{proposition}  The system (\ref{eq2}) around $E_{*}$ is
locally asymptotically stable if the condition (i) $k< \min{\{a_1+b_1 x_{2*},~ a_2+b_2 x_{3*}\}}$ is satisfied.\label{prop3}\end{proposition}

 {\textbf{Proof.}}
  Let $J(x_*)$ = ${(J_{ij})}_{3\times 3}$ is the Jacobian matrix at the interior equilibrium point $E_* =x_*$ of the system (\ref{eq2}). The components of $J(x_*)$ are $J_{11}=\frac{c_1 x_{1*} x_{2*}\bigl(k-(a_1+b_1 x_2)\bigr)}{(x_{1*}+k)(a_1+x_{1*} +b_1 x_{2*})} +\frac{c_2 x_{1*} x_{3*}\bigl(k-(a_2+b_2 x_3)\bigr)}{(x_{1*}+k)(a_2+x_{1*} +b_2 x_{3*})}$,
$J_{12}=-\frac{c_1 x_{1*}(a_1+x_{1*})}{(a_1 + x_{1*} + b_1 x_{2*})^2} < 0$, $J_{13}=-\frac{c_2 x_{1*}(a_2+x_{1*})}{(a_2 + x_{1*} + b_2 x_{3*})^2} < 0$, $J_{21}=\frac{c_1e_1 x_{2*}(a_1+b_1 x_{2*})}{(a_1 + x_{1*} + b_1 x_{2*})^2} > 0$, $J_{22}= -\frac{b_1c_1e_1 x_{1*} x_{2*}}{(a_1 + x_{1*} + b_1 x_{2*})^2} < 0$, $J_{23}=0$, $J_{31}=\frac{c_2e_2 x_{3*}(a_2+b_2 x_{3*})}{(a_2 + x_{1*} + b_2 x_{3*})^2} > 0$,  $J_{32}=0$,~$J_{33}= -\frac{b_2 c_2 e_2 x_{1*} x_{3*}}{(a_2 + x_{1*} + b_2 x_{3*})^2} < 0.$

Then the characteristic equation of the Jacobian matrix $J(x_*)$ can be written as
\begin{equation}\label{3.2h}
\lambda^3+k_1\lambda^2+k_2\lambda+k_3=0,
\end{equation}
where $k_1=-\mbox{tr}(J)=-(J_{11}+J_{22}+J_{33})$,
$k_2=M_{11}+M_{22}+M_{33}=(J_{11}J_{22}-J_{21}J_{12})+J_{22}J_{33}+(J_{11}J_{33}-J_{31}J_{13})$,
$k_3=-\det{(J)} =
-\bigl(J_{11}J_{22}J_{33}-J_{12}J_{21}J_{33}-J_{31}J_{13}J_{22}\bigr)$,
\mbox{and} ~$C_{2} = k_1 k_2 -k_3 = -(J_{11} +J_{22})\bigl(J_{33}(J_{11}+J_{22}+J_{33})+(J_{11}J_{22}-J_{21}J_{12})\bigr) +J_{13}J_{31}(J_{11} +J_{33}).$

It is clear that $k_{1} > 0$ if $J_{11}<0,$ i.e., $k< \min{\{a_1+b_1 x_{2*},~ a_2+b_2 x_{3*}\}}$ and consequently $C_{2} > 0.$ Hence the Routh-Hurwitz condition is satisfied for the matrix $J_{*}$, i.e., all the characteristic roots of $J_{*}$ are with negative real parts. So the system is locally asymptotically stable around $E_{*}$.

\begin{theorem}\label{theorem1}The dynamical system  (\ref{eq2}) undergoes Hopf bifurcation around the interior equilibrium point $E_{*}$ whenever the critical parameter value $r = r_{c}$ contained in the domain
 \begin{eqnarray}D_{HB} = \Bigl\{r_{c}\in {\mathbf{R}}^{+} :
C_{2}(r_c)=k_{1}(r_c)k_{2}(r_c)-k_{3}(r_c)=0 ~\mbox{with}~k_{2}(r_c) >0 ~\mbox{and}~ \frac{d C_{2}}{dr }|_{r =r_{c}}\neq 0\Bigr\}.\nonumber\end{eqnarray}
\end{theorem}

{\textbf{Proof.}}
The equation (\ref{3.2h}) will have a pair of purely imaginary
roots if $k_1k_2-k_3=0$ for some set of values of the system
parameters. Let us now suppose that $r=r_c$ be the value of $r$
satisfying the condition $k_1k_2-k_3=0$. Here only $J_{11}$ contains
$r$ explicitly. So, we write the equation $k_1k_2-k_3=0$ as an
equation in $J_{11}$ to find $r_c$ as follows:
\begin{equation}\label{3.8}
h_1J_{11}^2+h_2J_{11}+h_3=0,
\end{equation}
where $h_1=J_{22}+J_{33}$, ~$h_2=-J^2_{22}+J^2_{33}-J_{13}J_{31}-J_{12}J_{21}$, ~$h_3=(J_{22}+J_{33})J_{22}J_{33}-J_{13}J_{31}J_{33}-J_{12}J_{21}J_{22}$.

Thus, $J_{11}=\frac{1}{2h_1} (-h_2\pm
\sqrt{h^2_2-4h_1h_3})=J^*_{11}.$

Or,\begin{equation}\label{3.11} r=\frac{(x_{1*}+k)^{2}}{
k^2}\left[J^*_{11}+\frac{c_{1}
x_{2*}(a_{1}+b_{1}x_{2*})}{(a_1+x_{1*}+b_{1}
x_{2*})^2}-\frac{c_{2}x_{3*}(a_2+b_{2}x_{3*})}{(a_{2}+x_{1*}+b_{2}x_{3*})^{2}}\right]=r_c.
\end{equation}

Using the condition $k_1k_2-k_3=0,$ from equation (\ref{3.2h}) one can obtain
\begin{eqnarray} (\lambda+k_{1})(\lambda ^2+k_{2})=0,
\label{3.3h} \end{eqnarray} which has three roots $\lambda _{1}=+i\sqrt k_{2},$ ~ $\lambda_{2}=-i\sqrt k_{2},$ ~$\lambda _{3}=-k_{1},$ so there is a pair of purely imaginary eigenvalues $\pm i\sqrt k_{2}$.
For all values of $\lambda$, the roots are, in general, of the form $\lambda_{1}(r)=\xi _{1}(r)+i\xi _{2}(r ), ~ \lambda_{2}(r)=\xi_{1}(r )-i\xi_{2} (r),~ \lambda_{3}(r )=-k_{1}(r).$

Differentiating the characteristic equation (\ref{3.2h}) w.r.t. $r$, we have
\begin{eqnarray}
\frac{d \lambda}{d r} &=& -\frac{\lambda^2 \dot{k}_{1} + \lambda\dot{k}_{2}+\dot{k}_{3}}{3\lambda^2 + 2k_{1}\lambda +k_{2}}{\mid}_{\lambda = i \sqrt{k_2}}\nonumber\\
&=& \frac{\dot{k}_{3}-k_2 \dot{k}_{1} + i\dot{k}_{2}\sqrt{k_2}}{2(k_2 - ik_{1} \sqrt{k_2})}\nonumber\\
&=& \frac{\dot{k}_{3}-(\dot{k}_{1} k_2 +k_1 \dot{k}_{2})}{2(k_{1}^{2} +k_{2})}+i\frac{\sqrt{k_2}(k_1\dot{k}_{3}+k_2\dot{k}_{2} -k_1 \dot{k}_{1}k_2)}{2k_2(k_{1}^{2} +k_{2})}\nonumber\\&
=&-\frac{\frac{dC_2}{dr}}{2(k_{1}^{2} +k_{2})}+i\Bigl[\frac{\sqrt{k_2}\dot{k}_{2}}{2k_2} -\frac{k_1\sqrt{k_2}\frac{dC_2}{d r}}{2k_2(k_{1}^{2} +k_{2})}\Bigr].\label{ReIm}\hbox{~~~~~~~~~~~~~~~~~~~~~~~~~~}
\end{eqnarray}
Hence,
\begin{eqnarray}\frac{d}{dr}{\bigl(\mbox{Re}(\lambda(r))\bigr)}{\mid}_{r =
r_{c}}&=& -\frac{\frac{dC_2}{dr}}{2(k_{1}^{2} +k_{2})}{\mid}_{r =
r_{c}} \neq 0.\label{Re}
\end{eqnarray}

Using the monotonicity condition of the real part of the complex root $\frac{\mbox{d}{(\mbox{Re}(\lambda(r)))}}{\mbox{d}r }{\mid}_{r =r_{c}} \neq 0$ (cf. Wiggins \cite{wiggins}, pp. 380), one can easily establish the transversality condition $\frac{dC_2}{dr}|_{r =r_{c}}\neq 0,$ to ensure the existence of Hopf bifurcation around $E_{*}.$

\section{Global analysis of the system around the interior equilibrium}\label{Direction}
\subsection{Direction of Hopf bifucation of the system (\ref{eq2}) around  $E_{*}$}\label{subcriticalH}
In this Section we study on  the direction of Hopf bifucation around the interior equilibrium.
From the model equations (\ref{eq2}), we have
\begin{equation}\label{4.1}
    \dot{x}={f}(x),
\end{equation}
where ${x}=(x_1,x_2,x_3)^t$, ${f} =(f^1,f^2,f^3)^t=
   \left(\begin{array}{c}
    rx_1(1-\frac{x_1}{x_1+k})-\frac{c_{1}x_1 x_2}{a_{1}+x_1+b_{1}x_2}
                  -\frac{c_{2}x_1x_3}{a_{2}+x_1+b_{2}x_3} \nonumber\\
    -\delta_{1}x_2+\frac{c_{1}e_{1}x_1x_2}{a_{1}+x_1+b_{1}x_2}\nonumber\\
    -\delta_{2}x_3+\frac{c_{2}e_{2}x_1x_3}{a_{2}+x_1+b_{2}x_3} \\
  \end{array}\right)$.
Here, at $x=x_*$, $f=0$. Let ${y}= (y_1, y_2,
y_3)={(x_1-x_1*,~x_2-x_2*,~x_3-x_3*)}$. Putting in equation (\ref{4.1}), we have
\begin{equation}\label{4.2}
\dot{y}=J(x_{1*}, x_{2*},x_{3*}){y}+ \phi,
\end{equation}
where the components of nonlinear vector function $\phi=(\phi_1,\phi_2, \phi_3)^t$ are given by
\begin{equation}\label{4.6}
\phi_{i}=f^i_{x_1x_1}y^2_1+f^i_{x_2x_2}y^2_2+f^i_{x_3x_3}y^2_3+2f^i_{x_2x_3}y_2y_3+2f^i_{x_3x_1}y_3y_1+2f^i_{x_1x_2}y_1y_2+\mbox{h.o.t.},~
i=1,2,3.
\end{equation}

The coefficients of nonlinear terms in $y_{i},$ $i = 1,2,3$ are given by

\begin{math}
 f^{1}_{x_1x_1}=-\frac{2rk^2}{(x_1+k)^3}+\frac{2c_1x_2(a_1+b_1x_2)}{(a_1+x_1+b_{1} x_2)^3}
 +\frac{2c_2x_3(a_2+b_2x_3)}{(a_{2}+x_1+b_{2}x_3)^3},~ f^{1}_{x_2x_2}=\frac{2b_1c_1x_1(a_1+x_1)}{(a_1+x_1+b_{1}
 x_2)^3},~f^{1}_{x_3x_3}=\frac{2b_2c_2x_1(a_2+x_1)}{(a_2+x_1+b_{2}
 x_3)^3},\\
 f^{1}_{x_2 x_3}=0,~ f^{1}_{x_1 x_2}=-\frac{c_1 a_1(a_1 + x_1+b_1x_2)+ 2b_1c_1 x_1 x_2}{(a_1+x_1+b_{1}
 x_2)^3},~ f^{1}_{x_3x_1}=-\frac{c_2 a_2(a_2 + x_1+b_2x_3)+ 2b_2c_2 x_3 x_1}{(a_2+x_1+b_{2}
 x_3)^3};
  \end{math}
\begin{math}
f^{2}_{x_1x_1}= -\frac{2c_{1}e_{1}x_2(a_{1}+b_{1}
x_2)}{(a_{1}+x_1+b_{1} x_2)^3},~f^{2}_{x_2x_2}=-\frac{2b_1c_1e_1x_1(a_1+x_1)}{(a_1+x_1+b_{1}x_2)^3},~f^{2}_{x_3x_3}=0, ~
f^{2}_{x_1x_2}=\frac{a_1 c_1e_1(a_1+x_1+b_1x_2)+2b_1 c_1 e_1x_1 x_2}{(a_1+x_1+b_1x_2)^3},\\
f^{2}_{x_2x_3}=0,~f^{2}_{x_3x_1}=0;
\end{math}
\begin{math}
 f^{3}_{x_1x_1}=-\frac{2c_2e_2x_3(a_2+b_2x_3)}{(a_{2}+x_1+b_{2}
 x_3)^3},~ f^{3}_{x_2x_2}=0,~ f^{3}_{x_3x_3}=-\frac{2b_2c_2e_2x_1(a_2+x_1)}{(a_2+x_1+b_2x_3)^3},\\
 f^{3}_{x_2x_3}=0,~ f^{3}_{x_3x_1}=\frac{a_2c_2e_2(a_2+x_1+b_2x_3)+2b_2c_2e_2 x_1 x_3}{(a_2+x_1+b_2x_3)^3},~
 f^{3}_{x_1x_2}=0.
\end{math}

 Let $P$ be the matrix formed by the column vectors
$({\bf u_2, u_1,u_3}),$ which are the eigenvectors corresponding to the
eigenvalues $\lambda_{1,2}=\pm i\sqrt{k_2}$ and $\lambda_3=-k_1$
of $J(x_{1*}, x_{2*},x_{3*})$, then $J(x_{1*}, x_{2*},x_{3*}){\bf u_2}=i\sqrt{k_2}{\bf u_2},$~
$J(x_{1*},x_{2*},x_{3*}){\bf u_1}=-i\sqrt{k_2}{\bf u_1},$ and
$J(x_{1*}, x_{2*},x_{3*}){\bf u_3}=-k_1{\bf u_3}.$

Thus, $$P=\left(
\begin{matrix}
0   &  1 &  1\\
\frac{-J_{21}\sqrt{k_2}}{J^{2}_{22}+k_2} & \frac{-J_{21}J_{22}}{J^{2}_{22}+k_2}  & \frac{-J_{21}}{J_{22}+k_1}\\
\frac{-J_{31}\sqrt{k_2}}{J^{2}_{33}+k_2} &
\frac{-J_{31}J_{33}}{J^{2}_{33}+k_2}  &
\frac{-J_{31}}{J_{33}+k_1}\\
\end{matrix}\right)=(p_{ij})_{3\times3}.
$$

Let us make use of the transformation ${\bf y}=P{\bf z},$ so as the system
(\ref{4.2}) is reduced to the following one
\begin{equation}\label{4.10}
\dot{\bf z}=P^{-1}J(x_{1*}, x_{2*},x_{3*})P{\bf z}
+P^{-1}\phi=\left(
\begin{matrix}
0          & -i\sqrt{k_2} & 0\\
i\sqrt{k_2} & 0           & 0\\
0          & 0           & -k_1 \end{matrix}\right){\bf
z}+P^{-1}\phi.
\end{equation}
Here $P^{-1}=\frac{\mbox{Adj}P}{\det P}=(q_{ij})_{3\times3},$ where

 $q_{11}=\frac{1}{\det
P}(\frac{J_{21}J_{22}J_{31}}{(J^{2}_{22}+k_2)(J_{33}+k_1)}-
\frac{J_{21}J_{33}J_{31}}{(J^{2}_{22}+k_2)(J_{22}+k_1)}),~ q_{12}=\frac{1}{\det
P}( \frac{J_{31}}{J^{2}_{33}+k_1}- \frac{J_{31}J_{33}}{J^{2}_{33}+k_2}),\\
q_{13}=\frac{1}{\det
P}(\frac{-J_{21}}{J_{21}+k_1}+\frac{J_{21}J_{22}}{J^{2}_{22}+k_2}),~q_{21}=\frac{1}{\det
P}(\frac{J_{21}\sqrt{k_2}J_{31}}{(J^{2}_{22}+k_2)(J_{22}+k_1)}-
\frac{J_{21}\sqrt{k_2}J_{31}}{(J^{2}_{33}+k_2)(J_{22}+k_1)}),\\
q_{22}=\frac{1}{\det
P}\frac{\sqrt{k_2}J_{31}}{(J^{2}_{33}+k_2)},~ q_{23}=\frac{1}{\det P}\frac{-\sqrt{k_2}J_{31}}{(J^{2}_{33}+k_2)},~
q_{31}=\frac{1}{\det
P}(\frac{J_{21}\sqrt{k_2}J_{31}J_{33}}{(J^{2}_{22}+k_2)(J^{2}_{33}+k_2)}-
\frac{J_{21}\sqrt{k_2}J_{31}J_{22}}{(J^{2}_{33}+k_2)(J^{2}_{22}+k_2)}),\\
q_{32}=\frac{1}{\det
P}\frac{-\sqrt{k_2}J_{21}}{(J^{2}_{22}+k_2)},~ q_{33}=\frac{1}{\det P}\frac{\sqrt{k_2}J_{21}}{(J^{2}_{22}+k_2)}.$

The system (\ref{4.10}) can be written as
\begin{eqnarray}
    \frac{d}{dt}\left(\begin{array}{c}
      z_1 \\
      z_2 \\
    \end{array}\right)&=&\left(
\begin{matrix}
0          & -\sqrt{k_2} \\
\sqrt{k_2} & 0
\end{matrix}\right)\left(\begin{array}{c}
      z_1 \\
      z_2 \\
    \end{array}\right)+F(z_1,z_2,z_3),\label{4.15}\\
 \frac{dz_3}{dt}&=&-k_1z_3+G(z_1,z_2,z_3)\label{4.16}.
\end{eqnarray}
On the center-manifold (cf. Carr \cite{Carr}, Kar \cite{Kar})
\begin{equation}\label{4.21}
z_3=\frac{1}{2}\left(b_{11}z^2_1+2b_{12}z_1z_2+b_{22}z^2_2\right)
\end{equation}
Therefore, \begin{equation}\label{4.22}
\dot{z}_3=\left(\begin{array}{cc}
b_{11} z_1+b_{12}z_2 &  b_{12}z_1+b_{22}z_2 \\
\end{array}\right)\left(
\begin{matrix}
0          & -\sqrt{k_2} \\
\sqrt{k_2} & 0            \end{matrix}\right) \left(
  \begin{matrix}
  z_1 \\
  z_2
  \end{matrix}
  \right)\\
  =\sqrt{k_2} b_{12} z^2_1+\sqrt{k_2}(b_{22}-b_{11})z_1z_2-\sqrt{k_2}
  b_{12}z^2_2
\end{equation}
Using (\ref{4.10}) and (\ref{4.16}), we have

\begin{equation}\label{4.26}
\dot{z}_3=-k_1z_3+q_{31}\phi_1+q_{32}\phi_2+q_{33}\phi_3
\end{equation}
From the equations (\ref{4.22}) and (\ref{4.26}), we have
\begin{eqnarray}&~&\sqrt{k_2} b_{12} z^2_1+\sqrt{k_2}(b_{22}-b_{11})z_1z_2-\sqrt{k_2}
  b_{12}z^2_2\nonumber\\
  &~& = -\frac{1}{2}k_1(b_{11}z^2_1+2b_{12}z_1z_2+b_{22}z^2_2)
+q_{31} \bigl[f^1_{x_1x_1}\{p_{11}z_1+p_{12}z_2+p_{13}\frac{1}{2}(b_{11}z^2_1+2b_{12}z_1z_2+b_{22}z^2_2)\}^2\nonumber\\
&~&  +f^1_{x_2x_2}\{p_{21}z_1+p_{22}z_2+p_{23}\frac{1}{2}(b_{11}z^2_1+2b_{12}z_1z_2+b_{22}z^2_2)\}^2
+f^1_{x_3x_3}\{p_{31}z_1+p_{32}z_2+p_{33}\frac{1}{2}(b_{11}z^2_1\nonumber \\ &~&+2b_{12}z_1z_2+b_{22}z^2_2)\}^2
         +2f^1_{x_3x_1}\{p_{31}z_1+p_{32}z_2+p_{33}\frac{1}{2}(b_{11}z^2_1+2b_{12}z_1z_2+b_{22}z^2_2)\}\{p_{11}z_1+p_{12}z_2
         \nonumber\\ &~& +p_{13}\frac{1}{2}(b_{11}z^2_1+2b_{12}z_1z_2+b_{22}z^2_2)\}
         +2f^1_{x_1x_2}\{p_{11}z_1+p_{12}z_2+p_{13}\frac{1}{2}(b_{11}z^2_1+2b_{12}z_1z_2+b_{22}z^2_2)\}\nonumber\\&~&\times \{p_{21}z_1
         +p_{22}z_2+p_{23}\frac{1}{2}(b_{11}z^2_1+2b_{12}z_1z_2+b_{22}z^2_2)\}\bigr]\nonumber\\
         &~&+q_{32}[f^2_{x_1x_1}\{p_{11}z_1+p_{12}z_2+p_{13}\frac{1}{2}(b_{11}z^2_1+2b_{12}z_1z_2+b_{22}z^2_2)\}^2
           +f^2_{x_2x_2}\{p_{21}z_1+p_{22}z_2\nonumber\\&~& +p_{23}\frac{1}{2}(b_{11}z^2_1+2b_{12}z_1z_2+b_{22}z^2_2)\}^2
         +f^2_{x_3x_3}\{p_{31}z_1+p_{32}z_2+p_{33}\frac{1}{2}(b_{11}z^2_1+2b_{12}z_1z_2+b_{22}z^2_2)\}^2\nonumber\\&~&+2f^2_{x_2x_3}\{p_{21}z_1+p_{22}z_2
         +p_{23}\frac{1}{2}(b_{11}z^2_1+2b_{12}z_1z_2+b_{22}z^2_2)\}\{p_{31}z_1+p_{32}z_2+p_{33}\frac{1}{2}(b_{11}z^2_1\nonumber\\&~&+2b_{12}z_1z_2+b_{22}z^2_2)\}
         +2f^2_{x_3x_1}\{p_{31}z_1+p_{32}z_2+p_{33}\frac{1}{2}(b_{11}z^2_1+2b_{12}z_1z_2+b_{22}z^2_2)\}\{p_{11}z_1+p_{12}z_2\nonumber\\&~&+p_{13}\frac{1}{2}(b_{11}z^2_1
         +2b_{12}z_1z_2+b_{22}z^2_2)\}+2f^2_{x_1x_2}\{p_{11}z_1+p_{12}z_2+p_{13}\frac{1}{2}(b_{11}z^2_1+2b_{12}z_1z_2+b_{22}z^2_2)\}\nonumber\\&~&\times
         \{p_{21}z_1+p_{22}z_2+p_{23}\frac{1}{2}(b_{11}z^2_1+2b_{12}z_1z_2+b_{22}z^2_2)\}\bigr]\nonumber\\
&~&+q_{33}\bigl[f^3_{x_1x_1}\{p_{11}z_1+p_{12}z_2+p_{13}\frac{1}{2}(b_{11}z^2_1+2b_{12}z_1z_2+b_{22}z^2_2)\}^2+f^3_{x_2x_2}\{p_{21}z_1+p_{22}z_2
+p_{23}\nonumber\\&~&\times\frac{1}{2}(b_{11}z^2_1+2b_{12}z_1z_2+b_{22}z^2_2)\}^2
         +f^3_{x_3x_3}\{p_{31}z_1+p_{32}z_2+p_{33}\frac{1}{2}(b_{11}z^2_1+2b_{12}z_1z_2+b_{22}z^2_2)\}^2\nonumber\\&~&+2f^3_{x_2x_3}\{p_{21}z_1+p_{22}z_2
         +p_{23}\frac{1}{2}
         (b_{11}z^2_1+2b_{12}z_1z_2+b_{22}z^2_2)\}\{p_{31}z_1+p_{32}z_2+p_{33}\frac{1}{2}(b_{11}z^2_1\nonumber\\&~&+2b_{12}z_1z_2+b_{22}z^2_2)\}
         +2f^3_{x_3x_1}\{p_{31}z_1
         +p_{32}z_2+p_{33}\frac{1}{2}(b_{11}z^2_1+2b_{12}z_1z_2+b_{22}z^2_2)\}
         \{p_{11}z_1\nonumber\\&~&+p_{12}z_2+p_{13}\frac{1}{2}(b_{11}z^2_1+2b_{12}z_1z_2+b_{22}z^2_2)\}+2f^3_{x_1x_2}\{p_{11}z_1+p_{12}z_2+p_{13}\frac{1}{2}
         (b_{11}z^2_1\nonumber\\&~&+2b_{12}z_1z_2+b_{22}z^2_2)\}\{p_{21}z_1+p_{22}z_2+p_{23}\frac{1}{2}(b_{11}z^2_1+2b_{12}z_1z_2+b_{22}z^2_2)\}\bigr].\nonumber
         \end{eqnarray}
Comparing the coefficients of $z^2_1$, $z_1z_2$ and $z^2_2$ from
both sides, we have
\begin{eqnarray}\label{4.31}
&~&\sqrt{k_2} b_{12} +\frac{k_1}{2}b_{11}\nonumber\\&~&=q_{31}\left[ f^1_{x_1x_1}
p^2_{11}+f^1_{x_2x_2} p^2_{21}+f^1_{x_3x_3} p^2_{31}+2f^1_{x_2x_3}
p_{21}p_{31}+2f^1_{x_3x_1} p_{31}p_{11}+2f^1_{x_1x_2}
p_{11}p_{21}\right]\nonumber\\
&~&+q_{32}\left[ f^2_{x_1x_1} p^2_{11}+f^2_{x_2x_2}
p^2_{21}+f^2_{x_3x_3} p^2_{31}+2f^2_{x_2x_3}
p_{21}p_{31}+2f^2_{x_3x_1} p_{31}p_{11}+2f^2_{x_1x_2}
p_{11}p_{21}\right]\nonumber\\
&~&+q_{33}\left[ f^3_{x_1x_1} p^2_{11}+f^3_{x_2x_2}
p^2_{21}+f^3_{x_3x_3} p^2_{31}+2f^3_{x_2x_3}
p_{21}p_{31}+2f^3_{x_3x_1} p_{31}p_{11}+2f^3_{x_1x_2}
p_{11}p_{21}\right]\nonumber\\
&~&=\Omega_1,
\end{eqnarray}
\begin{eqnarray}\label{4.32}
&~&-\sqrt{k_2} (b_{11}-b_{22})+k_1 b_{12}\nonumber\\ &~& = 2q_{31}[ f^1_{x_1x_1}
p_{11}p_{12}+f^1_{x_2x_2} p_{21}p_{22}+f^1_{x_3x_3} p_{31}p_{32}+f^1_{x_2x_3} (p_{21}p_{32}+p_{22}p_{31}) +
f^1_{x_3x_1} (p_{11}p_{32}\nonumber\\ &~& +p_{12}p_{31})+f^1_{x_1x_2}(p_{11}p_{22}+p_{12}p_{21})]+q_{32}[f^2_{x_1x_1} p_{11}p_{12}+f^2_{x_2x_2}
p_{21}p_{22}+f^2_{x_3x_3}p_{31}p_{32}  +f^2_{x_2x_3}\nonumber\\ &~&\times (p_{21}p_{32}+p_{22}p_{31}) + f^2_{x_3x_1}
(p_{11}p_{32}+p_{12}p_{31})+f^2_{x_1x_2}(p_{11}p_{22}+p_{12}p_{21})]+q_{33}\left[f^3_{x_1x_1} p_{11}p_{12}\right.\nonumber\\ &~& \left.+f^3_{x_2x_2}
p_{21}p_{22}+f^3_{x_3x_3} p_{31}p_{32}+f^3_{x_2x_3} (p_{21}p_{32}+p_{22}p_{31}) + f^3_{x_3x_1}
(p_{11}p_{32}+p_{12}p_{31})\right.\nonumber \\ &~& \left.+f^3_{x_1x_2}(p_{11}p_{22}+p_{12}p_{21})\right]=\Omega_2,
\end{eqnarray}
and
\begin{eqnarray}\label{4.33}
&~& -\sqrt{k_2} b_{12}+\frac{k_1}{2}b_{22}\nonumber\\&~&=q_{31}\left[ f^1_{x_1x_1}
p^2_{12}+f^1_{x_2x_2} p^2_{22}+f^1_{x_3x_3} p^2_{32}+2f^1_{x_2x_3}
p_{22}p_{32}+2f^1_{x_3x_1} p_{31}p_{12}+2f^1_{x_1x_2}
p_{12}p_{22}\right]\nonumber\\
&~&+q_{32}\left[ f^2_{x_1x_1} p^2_{12}+f^2_{x_2x_2}
p^2_{22}+f^2_{x_3x_3} p^2_{32}+2f^2_{x_2x_3}
p_{22}p_{32}+2f^2_{x_3x_1} p_{31}p_{12}+2f^2_{x_1x_2}
p_{12}p_{22}\right]\nonumber\\
&~&+q_{33}\left[ f^3_{x_1x_1} p^2_{12}+f^3_{x_2x_2}
p^2_{22}+f^3_{x_3x_3} p^2_{32}+2f^3_{x_2x_3}
p_{22}p_{32}+2f^3_{x_3x_1} p_{31}p_{12}+2f^3_{x_1x_2}
p_{12}p_{22}\right]\nonumber\\
&~&=\Omega_3.
\end{eqnarray}
From equations (\ref{4.31}), (\ref{4.32}) and (\ref{4.33}), we
have
\begin{eqnarray}\label{4.35}
\left(%
\begin{array}{ccc}
  \frac{1}{2}k_1 & \sqrt{k_2} & 0 \\
  -\sqrt{k_2} & k_1 & \sqrt{k_2}\\
  0 & -\sqrt{k_2} & \frac{1}{2}k_1 \\
\end{array}
\right)\left(
\begin{array}{c}
  b_{11} \\
  b_{12} \\
  b_{22}
\end{array}
\right)=\left(
\begin{array}{c}
  \Omega_1 \\
  \Omega_2 \\
  \Omega_3
\end{array}
\right).
\end{eqnarray}
 The equation
(\ref{4.35}) gives the coefficients $b_{11}$, $b_{12}$ and
$b_{22}$ as follows:
\begin{eqnarray}
b_{11} &=& \frac{k_2(\Omega_1+\Omega_3) -\frac{k_1}{2}(\sqrt{k_2}\Omega_2 -k_1\Omega_1)}{(\frac{k_1^3}{4} +k_1 k_2)},\nonumber\\
b_{12} &=& \frac{\frac{k_1^2 \Omega_2}{4} -\frac{k_1 \sqrt{k_2}}{2}(\Omega_3 -\Omega_1)}{(\frac{k_1^3}{4} +k_1 k_2)},\nonumber\\
b_{22} &=&\frac{k_2(\Omega_1+\Omega_3) +\frac{k_1^2 \Omega_3}{2} + \frac{k_1\sqrt{k_2}}{2}\Omega_2}{(\frac{k_1^3}{4} +k_1 k_2)}.\nonumber
\end{eqnarray}
The flow of the central manifold is characterized by the reduced system as
\begin{equation}\label{4.40}
    \frac{d}{dt}\left(\begin{array}{c}
      z_1 \\
      z_2
    \end{array}\right)=\left(
\begin{matrix}
0          & -\sqrt{k_2} \\
\sqrt{k_2} & 0
\end{matrix}\right)\left(\begin{array}{c}
      z_1 \\
      z_2
    \end{array}\right)+\left(\begin{array}{c}
      F^1 \\
      F^2
    \end{array}\right),
\end{equation}
where $F^1=q_{11}\phi_1+q_{12}\phi_2+q_{13}\phi_3+h.o.t$,
$F^2=q_{21}\phi_1+q_{22}\phi_2+q_{23}\phi_3+h.o.t$. The stability
of the bifurcating limit cycle can be determined by the sign of
the parametric expression
\begin{eqnarray}\Pi = F^1_{111}+F^2_{112}+F^1_{122}+F^2_{222}+\frac{F^1_{12}(F^1_{11}+F^1_{22})-F^2_{12}(F^2_{11}+F^2_{22})-F^1_{11}F^2_{11}+F^1_{22}F^2_{22}}{\sqrt{k_2}},\label{eqpi}\end{eqnarray}
where $F_{ijk}=\frac{\partial^3 F}{\partial z_i\partial z_j
\partial z_k}$ at the origin.
If the value of the above expression is negative, then the Hopf bifurcating limit cycle is
stable and is called a supercritical Hopf bifurcation. If the value is
positive, then the Hopf bifurcating limit cycle is unstable and
the bifurcation is subcritical.

Here
\begin{eqnarray}F^1_{11}&=&2q_{11}[f^1_{x_1x_1}p^2_{11} + f^1_{x_2x_2}p^2_{21}+ f^1_{x_3x_3}p^2_{31}+2f^1_{x_3x_1}p_{11}p_{31} +2f^1_{x_1x_2}p_{11}p_{21}]+2q_{12}[f^2_{x_1x_1}p^2_{11} + f^2_{x_2x_2}\nonumber\\ &~& \times p^2_{21}+2f^2_{x_1x_2}p_{11}p_{21}]+2q_{13}[f^3_{x_1x_1}p^2_{11} + f^3_{x_3x_3}p^2_{31}+2f^3_{x_3x_1}p_{11}p_{31}],\nonumber \\
 F^1_{12}&=&2q_{11}[f^1_{x_1x_1}p_{11}p_{12} + f^1_{x_2x_2}p_{21}p_{22}+ f^1_{x_3x_3}p_{31}p_{32}+f^1_{x_3x_1}(p_{11}p_{32}+p_{12}p_{31}) +f^1_{x_1x_2}(p_{12}p_{21}+ \nonumber\\ &~&p_{22}p_{11})]+2q_{12}[f^2_{x_1x_1}p_{11}p_{12} + f^2_{x_2x_2}p_{21}p_{22}+ f^2_{x_1x_2}(p_{12}p_{21}+p_{22}p_{11})]
+2q_{13}[f^3_{x_1x_1}p_{11}p_{12} \nonumber\\ &~&+ f^3_{x_3x_3}p_{31}p_{32}+f^3_{x_3x_1}(p_{11}p_{32}+p_{12}p_{31})],\nonumber \\ F^1_{22}&=&2q_{11}[f^1_{x_1x_1}p_{12}^2 + f^1_{x_2x_2}p_{22}^2+ f^1_{x_3x_3}p_{32}^2+ 2f^1_{x_3x_1}p_{12}p_{32} +2f^1_{x_1x_2}p_{12}p_{22}] +2q_{12}[f^2_{x_1x_1}p_{12}^2 \nonumber\\ &~&+ f^2_{x_2x_2}p_{22}^2+ 2f^2_{x_1x_2}p_{12}p_{22}]
+2q_{13}[f^3_{x_1x_1}p_{12}^2 + f^3_{x_3x_3}p_{32}^2+f^3_{x_3x_1}p_{12}p_{32})],\nonumber \\
F^1_{111}&=&6q_{11}b_{11}[f^1_{x_1x_1}p_{11}p_{13}+f^1_{x_2x_2}p_{21}p_{23}+f^1_{x_3x_3}p_{31}p_{33}
+f^1_{x_3x_1}(p_{11}p_{33}+p_{13}p_{31})+f^1_{x_1x_2}(p_{11}p_{23}\nonumber\\ &~&+p_{21}p_{13})]
+6q_{12}b_{11}[f^2_{x_1x_1}p_{11}p_{13}+f^2_{x_2x_2}p_{21}p_{23}+f^2_{x_3x_3}p_{31}p_{33}
+f^2_{x_1x_2}(p_{11}p_{23}+p_{13}p_{21})]
\nonumber\\ &~&+6q_{13}b_{11}[f^3_{x_1x_1}p_{11}p_{13}+f^3_{x_3x_3}p_{31}p_{33}
+f^3_{x_3x_1}(p_{11}p_{33}+p_{13}p_{31})],\nonumber \\ F^1_{122}&=&2q_{11}[f^1_{x_1x_1}(2p_{12}p_{13}b_{12}+p_{11}p_{13}b_{22})+f^1_{x_2x_2}(2p_{22}p_{33}b_{12}+p_{21}p_{23}b_{22})
+f^1_{x_3x_3}(2p_{32}p_{33}b_{12}\nonumber\\ &~& +p_{31}p_{33}b_{22}) + f^1_{x_3x_1}(2p_{13}p_{32}b_{12}+p_{11}p_{33}b_{22}+p_{13}p_{31}b_{22}+2p_{12}p_{33}b_{12})
+ f^1_{x_1x_2}(2p_{12}p_{23}\nonumber \\ &~& \times b_{12}  +p_{13}p_{21}b_{22}+p_{11}p_{23}b_{22}+2p_{22}p_{13}b_{12})]
 +2q_{12}[f^2_{x_1x_1}(2p_{12}p_{13}b_{12}+p_{11}p_{13}b_{22}) +f^2_{x_2x_2}\nonumber\\ &~&\times (2p_{22}p_{23}b_{12}+p_{21}p_{23}b_{22})
+f^2_{x_1x_2}(2p_{12}p_{23}b_{12}+p_{11}p_{23}b_{22}+2p_{13}p_{22}b_{12}+p_{13}p_{21}b_{22})]\nonumber \\
&~& + 2q_{13}[f^3_{x_1x_1}(2p_{13}p_{12}b_{12}+p_{11}p_{13}b_{22})+f^3_{x_3x_3}(2p_{32}p_{33}b_{12}+p_{31}p_{33}b_{22})
+f^3_{x_3x_1}(2p_{32}p_{13}b_{12} \nonumber \\ &~& +p_{31}p_{13}b_{22}+2p_{12}p_{13}b_{12}+p_{33}p_{11}b_{22})], \nonumber\end{eqnarray}

\begin{eqnarray}F^2_{11}&=&2q_{21}[f^1_{x_1x_1}p^2_{11} + f^1_{x_2x_2}p^2_{21}+ f^1_{x_3x_3}p^2_{31}+2f^1_{x_3x_1}p_{11}p_{31} +2f^1_{x_1x_2}p_{11}p_{21}]+2q_{22}[f^2_{x_1x_1}p^2_{11} + f^2_{x_2x_2}\nonumber\\ &~& \times p^2_{21}+2f^2_{x_1x_2}p_{11}p_{21}]+2q_{23}[f^3_{x_1x_1}p^2_{11} + f^3_{x_3x_3}p^2_{31}+2f^3_{x_3x_1}p_{11}p_{31}],\nonumber \\
 F^2_{12}&=&2q_{21}[f^1_{x_1x_1}p_{11}p_{12} + f^1_{x_2x_2}p_{21}p_{22}+ f^1_{x_3x_3}p_{31}p_{32}+f^1_{x_3x_1}(p_{11}p_{32}+p_{12}p_{31}) +f^1_{x_1x_2}(p_{12}p_{21}+ \nonumber\\ &~&p_{22}p_{11})]+2q_{22}[f^2_{x_1x_1}p_{11}p_{12} + f^2_{x_2x_2}p_{21}p_{22}+ f^2_{x_1x_2}(p_{12}p_{21}+p_{22}p_{11})]
+2q_{23}[f^3_{x_1x_1}p_{11}p_{12} \nonumber\\ &~&+ f^3_{x_3x_3}p_{31}p_{32}+f^3_{x_3x_1}(p_{11}p_{32}+p_{12}p_{31})],\nonumber \\ F^2_{22}&=&2q_{21}[f^1_{x_1x_1}p_{12}^2 + f^1_{x_2x_2}p_{22}^2+ f^1_{x_3x_3}p_{32}^2+ 2f^1_{x_3x_1}p_{12}p_{32} +2f^1_{x_1x_2}p_{12}p_{22}] +2q_{22}[f^2_{x_1x_1}p_{12}^2 \nonumber\\ &~&+ f^2_{x_2x_2}p_{22}^2+ 2f^2_{x_1x_2}p_{12}p_{22}]
+2q_{23}[f^3_{x_1x_1}p_{12}^2 + f^3_{x_3x_3}p_{32}^2+f^3_{x_3x_1}p_{12}p_{32})],\nonumber\\
F^2_{112}&=&2q_{21}[f^1_{x_1x_1}(2p_{11}p_{13}b_{12}+p_{12}p_{13}b_{11})+f^1_{x_2x_2}(2p_{21}p_{23}b_{12}+p_{22}p_{23}b_{11})
+f^1_{x_3x_3}(2p_{31}p_{33}b_{12}\nonumber\\ &~& +p_{32}p_{33}b_{11}) + f^1_{x_3x_1}(2p_{11}p_{33}b_{12}+p_{13}p_{32}b_{11}+p_{12}p_{33}b_{11}+2p_{13}p_{31}b_{12})
+ f^1_{x_1x_2}(2p_{13}p_{21}\nonumber \\ &~& \times b_{12}  +p_{12}p_{23}b_{11}+ 2p_{11}p_{23}b_{12}+p_{22}p_{13}b_{11})]
 +2q_{22}[f^2_{x_1x_1}(2p_{11}p_{13}b_{12}+p_{12}p_{13}b_{11}) +f^2_{x_2x_2}\nonumber\\ &~&\times (2p_{21}p_{23}b_{12}+p_{22}p_{23}b_{11})
+f^2_{x_1x_2}(2p_{11}p_{23}b_{12}+p_{12}p_{23}b_{11}+2p_{13}p_{21}b_{12}+p_{13}p_{22}b_{11})]\nonumber \\
&~& + 2q_{23}[f^3_{x_1x_1}(2p_{11}p_{13}b_{12}+p_{12}p_{13}b_{11})+f^3_{x_3x_3}(2p_{31}p_{33}b_{12}+p_{32}p_{33}b_{11})
+f^3_{x_3x_1}(2p_{31}p_{13}b_{12} \nonumber \\ &~& +p_{33}p_{12}b_{11}+2p_{11}p_{33}b_{12}+p_{32}p_{13}b_{11})], \nonumber\\
F^2_{222}&=&6q_{21}b_{22}[f^1_{x_1x_1}p_{12}p_{13}+f^1_{x_2x_2}p_{22}p_{23}+f^1_{x_3x_3}p_{32}p_{33}
+f^1_{x_3x_1}(p_{12}p_{33}+p_{13}p_{32})+f^1_{x_1x_2}(p_{12}p_{23}\nonumber\\ &~&+p_{13}p_{22})]
+6q_{22}b_{22}[f^2_{x_1x_1}p_{12}p_{13}+f^2_{x_2x_2}p_{22}p_{23}+f^2_{x_1x_2}(p_{12}p_{23}+p_{13}p_{22})]
\nonumber\\ &~&+6q_{23}b_{22}[f^3_{x_1x_1}p_{12}p_{13}+f^3_{x_3x_3}p_{32}p_{33}
+f^3_{x_3x_1}(p_{12}p_{33}+p_{13}p_{32})].\nonumber \end{eqnarray}
\subsection{Global stability of the system (\ref{eq2}) around  $E_{*}$}\label{Globalint}
 \begin{theorem}
The interior equilibrium $E_{*}$ is globally asymptotically stable if the condition
\begin{eqnarray}
(\mbox{i})~ a_1 a_2 b_1 b_2 r k> (x_{1*} +k)(w +k)(a_1 b_1 c_2 + a_2 b_2 c_1) ~\hbox{is satisfied}.\nonumber\end{eqnarray}\label{theorem2}\end{theorem}
\textbf{Proof.}
 Let \begin{eqnarray} L(x_1,x_2,x_3)=L_{1}(x_1,x_2,x_3)+L_{2}(x_1,x_2,x_3)+L_{3}(x_1,x_2,x_3)\label{eqgas1}\end{eqnarray}
be a positive Lyapunov function, where
 \begin{eqnarray}
 ~L_{1}=s_{1}\bigl(x_1-x_{1*}-x_{1*}\ln(\frac{x_1}{x_{1*}}) \bigr),~ L_{2} =s_{2}\bigl(x_2-x_{2*}-x_{2*}\ln(\frac{x_2}{x_{2*}}) \bigr),\nonumber\\~L_{3}=s_{3}\bigl(x_3-x_{3*}-x_{3*}\ln(\frac{x_3}{x_{3*}})\bigr);\nonumber
\end{eqnarray}
 $s_{1},~s_{2}$ and $s_{3}$ being positive real constants.

This function is well-defined and continuous in Int(${R_{+}}^3$). It can be easily verified that the function $L(x_1,x_2,x_3)$ is zero at the equilibrium point $E_{*}$ and is positive for all other positive values of $(x_1,x_2,x_3),$~and thus $E_{*}$ is the global minimum of $L(x_1,x_2,x_3)$.

Since the solutions of the system are bounded and ultimately enter the set ${\Omega}=\{(x_1,x_2,x_3);x_1>0, x_2>0, x_3>0: x_1+\frac{x_2}{e_1} + \frac{x_3}{e_2}\leq M+\epsilon ,~\forall~\epsilon > 0 \}$, we restrict our study in ${\Omega}$. The time derivative of $L$ along with the  solutions of the system (\ref{eq2}) gives (cf. Sarwardi et al. \cite{Sarwardirefuge}, ~\cite{sarwardiglobal})
\begin{eqnarray}
\frac{dL}{dt}&=&-s_{1}\Bigl[\frac{rk}{(x_{1*}+k)(x_{1}+k)}-\frac{c_1 x_{2*}}{(a_1 + x_1 +b_1 x_2)(a_1 + x_{1*} +b_1 x_{2*})}\nonumber\\ &-&\frac{c_2 x_{3*}}{(a_2 + x_1 +b_2 x_3)(a_2 + x_{1*} +b_2 x_{3*})}\Bigr](x_1-x_{1*})^2- s_{2}\Bigl[\frac{c_1e_1(a_1+b_1)}{(a_1 + x_1 +b_1 x_2)(a_1 + x_{1*} +b_1 x_{2*})}\Bigr]\nonumber\\&\times&(x_2-x_{2*})^2 -s_{3}\Bigl[\frac{c_2e_2(a_2+b_2)}{(a_2 + x_1 +b_2 x_3)(a_2 + x_{1*} +b_2 x_{3*})}\Bigr](x_3-x_{3*})^2\nonumber\\&+&\frac{c_1\bigl(s_2 b_1e_1 x_{2*} - s_1(a_1+x_{1*})\bigr)(x_1-x_{1*})(x_2-x_{2*})}{(a_1 + x_1 +b_1 x_2)(a_1 + x_{1*} +b_1 x_{2*})} +\frac{c_2\bigl(s_3 b_2e_2 x_{3*} - s_1(a_2+x_{1*})\bigr)}{(a_2 + x_1 +b_2 x_2)(a_2 + x_{1*} +b_2 x_{2*})}\nonumber\\ &\times& (x_1-x_{1*})(x_3-x_{3*}). \label{gas2}
\end{eqnarray}
Letting $ s_{1} = 1$, ~$ s_{2} = \frac{a_1+x_{1*}}{b_1 e_1 x_{2*})}$ and  $s_{3} =\frac{a_2+x_{1*}}{b_2 e_2 x_{3*})},$ we have
\begin{eqnarray}
\frac{dL}{dt}&\leq&-\Bigl[\frac{rk}{(x_{1*}+k)(x_{1}+k)}-\frac{c_1 x_{2*}}{(a_1 + x_1 +b_1 x_2)(a_1 + x_{1*} +b_1 x_{2*})}\nonumber\\ &-&\frac{c_2 x_{3*}}{(a_2 + x_1 +b_2 x_3)(a_2 + x_{1*} +b_2 x_{3*})}\Bigr](x_1-x_{1*})^2\nonumber\\
&<&-\Bigl[\frac{rk}{(x_{1*}+k)(w+k)} -\frac{c_1}{a_1b_1}-\frac{c_2}{a_2b_2}\Bigr](x-x_{*})^2 \nonumber\\
&<& 0, \hbox{ by condition (\mbox{i})},
\end{eqnarray}

along all the trajectories in the positive octant except $(x_{1*}, x_{2*}, x_{3*})$. Also $\frac{dL}{dt} = 0$ when $(x_1,x_2,x_3) = (x_{1*}, x_{2*}, x_{3*})$.
The proof follows from (\ref{eqgas1}) and Lyapunov–-Lasalle's
invariance principle (cf. Hale \cite{Hale}).

\begin{table}[ht]
\caption{Schematic representation of our analytical findings: LAS = Locally
asymptotically stable, GAS = Globally asymptotically stable, HB = Hopf bifurcation, SHB = Subcritical Hopf bifurcation.}
\begin{center}
\begin{tabular}{|c|c|c|c|}\hline
{\textbf{Equilibria}}& \parbox[t]{1.80in}{\textbf{~~Feasibility conditions/ \,parametric restrictions~}}& \parbox[t]{1.80in}{\textbf{~~Stability conditions/ \,parametric restrictions~}}& \textbf{Nature} \\
[3ex]
\hline
$E_{0}$ & {No Conditions} & {No Conditions} & {Unstable}  \\
[1.5ex] \hline
$E_{1}$ & {$c_1e_1 >\delta_1$,~ $x_{1_1}> \frac{b_1 \delta_1}{c_1 e_1-\delta_1}$} & {$e_1 x_{1_1} + x_{2_1}>\frac{k(1-b_1 e_1)-a_1}{b_1}$} & {LAS}  \\
[1.5ex] \hline
$E_{2}$ & {$c_2e_2 >\delta_2$,~ $x_{1_2}> \frac{b_2 \delta_2}{c_2 e_2-\delta_2}$} & {$e_2 x_{1_2} + x_{3_2}>\frac{k(1-b_2 e_2)-a_2}{b_2}$} & {LAS}  \\
[1.5ex] \hline
$E_{*}$ & $x_{1*}>\max
\left\{\frac{a_1\delta_1}{c_1e_1-\delta_1},
\frac{a_2\delta_2}{c_2e_2-\delta_2} \right\}$   & {$k< \min{\{a_1+b_1 x_{2*},~\linebreak a_2+b_2 x_{3*}\}}$} &{LAS}  \\
[0.5ex] \hline
$E_{*}$ & {\ldots\ldots\ldots\ldots\ldots\ldots\ldots} &\parbox[t]{1.80in} {$(i)\, r > \delta_1 + \delta_{2},~
(ii) ~ x_{1_{1}} >\frac{a_{2}\delta_{2}}{c_{2} e_{2} -\delta_{2}}, \, (iii) \, x_{1_{2}} >\frac{a_{1}\delta_{1}}{c_{1} e_{1} -\delta_{1}}$} & {Persistence} \\
[1.5 ex] \hline
$E_{*}$ & {\ldots\ldots\ldots\ldots\ldots\ldots\ldots} & {Stated in the Proposition \ref{prop2}} & {Boundedness} \\
[1.5 ex] \hline
$E_{*}$ & {\ldots\ldots\ldots\ldots\ldots\ldots\ldots} & $\Pi>0$ {(cf. equation (\ref{eqpi}))}  & {SHB}  \\
[1.5 ex] \hline
$E_{*}$ & $x_{1*}>\max
\left\{\frac{a_1\delta_1}{c_1e_1-\delta_1},
\frac{a_2\delta_2}{c_2e_2-\delta_2} \right\}$ & \parbox[t]{2.0in}{$a_1 a_2 b_1 b_2 r k> (x_{1*} +k)(w +k)\times(a_1 b_1 c_2 + a_2 b_2 c_1)$} & {GAS}  \\
[1.5ex] \hline
\end{tabular}\label{Table1}
\end{center}
\end{table}
\begin{table}
\centering \caption{The set of system parameter (including the critical parameter $r_{c}$) values and their corresponding Figures with description.}
\begin{tabular}{ |c| c| c | c |c|c|}\hline
  %\begin{tabular}{|c|r|c|l|l} \hline
  %\begin{tabular}{l|lp{2in}} \hline
  %\begin{tabular}{{lcc{2in}} \hline
  %\begin{tabular}{|c|c|c|c{2in}} \hline
No. & \parbox[t]{1.5in}{Fixed Parameters}& {$ r $}& {Figures} &{Description} \\[3ex] \hline
 1 &\parbox[t]{2.2in}{$r= 1.37 > r_{c} = 1.320961640, k = 200, a_1 = 100, a_2= 100, b_1 = 0.5, b_2 = 0.5; c_1 = 1.8, c_2 = 1.8, \delta_1 = 0.82, \delta_2 = 0.62, e_1 = 0.8143, e_2 = 0.6250.$.}  & {$1.37$}& \parbox[t]{0.55in}{Figs. 1: (a)-(b)} & \parbox[t]{01.2in}{2D view of Hopf bifurcation}\\[2.5ex] \hline
 2&${\ldots\ldots\ldots\ldots\ldots\ldots\ldots}$ & \parbox[t]{1.0in}{$1.320961640$}& \parbox[t]{0.55in}{Figs. 2} & {Limit cycle}\\[1.5ex] \hline
 3& ${\ldots\ldots\ldots\ldots\ldots\ldots\ldots}$  &\parbox[t]{1.0in}{$r\in [0.8, 2.0]$} &\parbox[t]{0.5in}{Fig. 3}& \parbox[t]{01.35in}{Hopf bifurcation (growth rate $r$ vs. population volumes)} \\[1.5 ex] \hline
 4 & ${\ldots\ldots\ldots\ldots\ldots\ldots\ldots}$  &\parbox[t]{1.0in}{$1.4700000000$} & \parbox[t]{0.5in}{Fig. 4} &\parbox[t]{01.2in}{2D view of local stability} \\[0.5ex] \hline
  5& ${\ldots\ldots\ldots\ldots\ldots\ldots\ldots}$ &  \parbox[t]{1.0in}{$1.4700000000$} & \parbox[t]{0.5in}{Fig. 5} & \parbox[t]{01.2in}{3D view of local stability}  \\[0.5ex] \hline
   6& ${\ldots\ldots\ldots\ldots\ldots\ldots\ldots}$  &\parbox[t]{1.0in}{$1.4700000000$} & \parbox[t]{0.5in}{Fig. 6} & {Global stability} \\[0.5ex] \hline
   7& \parbox[t]{2.2in}{$r= 1.37, k=200, a_1= 100, a_2=100, b_1=0.5; b_2=0.5, c_1=1.8, c2=1.8, \delta_1=0.82, \delta_2=0.62, e_1=0.8143, e_2=0.6250$}  &\parbox[t]{1.2in}{$r_c = 1.320961640,$ ~$\Pi=1.0424314050$} & \parbox[t]{0.55in}{Figs. 7} & \parbox[t]{01.2in}{Subcritical Hopf bifurcation} \\[0.5ex] \hline
\end{tabular}\label{Table2}\\[1.5ex]
% \centerline{\textrm{\small {\bf Table 1:} The set of parameter values and their corresponding figures.}}
 %\caption{\small{\bf Table 1:} The set of parameter values and their corresponding figures.}
% \end{center}
\end{table}

\section{Numerical simulation}\label{Numerical}
  For the purpose of making qualitative analysis of the present study, numerical simulations have been carried out by making use of MATLAB-R2010a and Maple-12. The analytical findings of the present study are summarized and represented schematically in Table \ref{Table1}. These results are all verified by means of numerical illustrations of which some chosen ones are exhibited in the figures. Here, we have given some numerical simulations on the study of stability and bifurcation of the proposed system (\ref{eq2}) around the interior equilibrium $E_{*}.$ We took a set of admissible parameter values: $r=1.7, k=200, a_1=a_2=100, b_1=b_2=0.5, c_1=c_2 = 1.8, \delta_1 =0.82, \delta_2=0.62, e_1=0.8143, e_2=0.6250.$ For this set of parameter values, it is found that the system possessed an unique interior equilibrium point $E_{*} = (169.1663564,55.36073780, 62.98120968).$ The system parameter $r$ is the growth rate of the prey population which plays a crucial role in regulating the dynamical behaviour of the proposed system. For this reason, we take this parameter as an influential parameter and try to determine the possible outcomes by varying this parameter within its feasible range. The interior equilibrium $E_*$ is stable for the values of $r>r_c=1.320961640$ (cf. Figure: \ref{f4}-\ref{f5} for local stability and Figure: \ref{f6} for global stability). The system (\ref{eq2}) experiences Hopf bifurcation when the parameter $r$ crosses the critical value $r_c$ from left to right, i.e., when $r=r_c$,  all the species coexist in the form of periodic oscillation. Following the steps discussed in Subsection \ref{subcriticalH}, we have found the value of $\Pi=1.042431405>0,$ which indicates that the the obtained Hopf bifurcation is subcritical bifurcation (cf. Figure \ref{f7}).

  It is observed that, if the interference coefficient $b_1$ (interference effect due to the presence of second predator on the first predator) increases it stabilize the system for $b_1=0.7$ while it is unstable at $b_1=0.6$ and when the parameter $b_1$ exceeds its value 20, the first predator population is died out from the system, i.e., the system breakdown. Similarly, the interference effect due to the presence of first predator on the second predator, parameterized by $b_2$ plays an important role to stabilize the system. If the interference coefficient $b_2$ increases it stabilizes the system for $b_2=0.6,$ while it is unstable at $b_2=0.5.$ It also regulates the existence of second predator in the system. As the parameter $b_2$ exceeds its value 20.9, the second predator population is died out from the system.

  Analogously, if the parameter $\delta_1,$ denoting the death rate of first predator increases then the volume of the fist predator decreases as well as second predator population increases and if $\delta_1$ decreases, the first predator population increases and second predator population decreases. If the death rate $\delta_1$ is gradually increased to a certain level the first predator population goes into extinction. Similar result is observed for the case of the second predator's death rate. The above observations ensure that the model under consideration is consistent with biological phenomenon (Figures are not reported here).

  \begin{figure}[h]
\centering
\begin{tabular}{cc}
(a) \epsfig{file=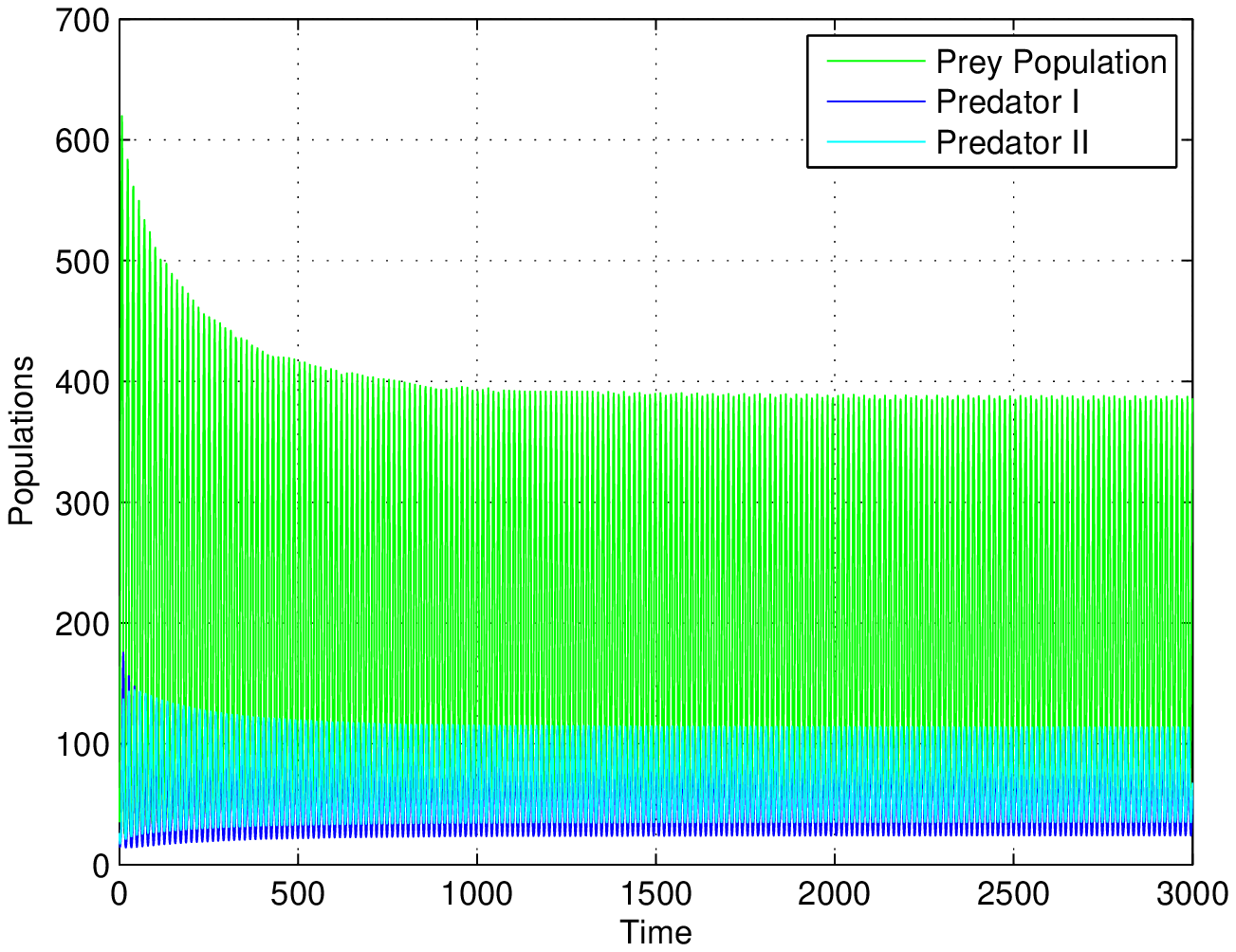,width=0.4\linewidth,height=5cm,clip=} &
(b) \epsfig{file=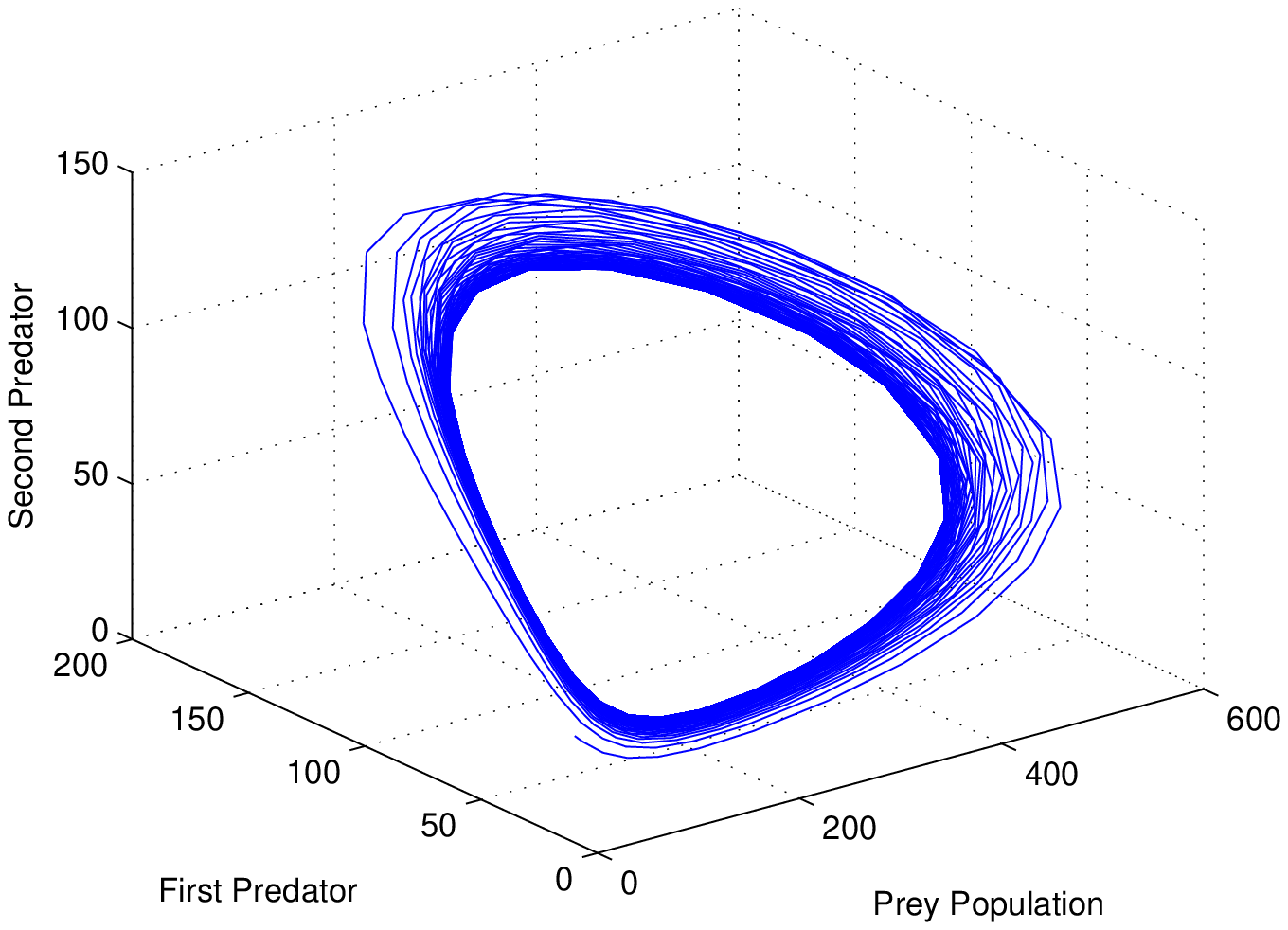,width=0.4\linewidth,height=5cm,clip=}
\end{tabular}
 \caption{\textrm{\small 2D view of Hopf bifurcation around the interior equilibrium $E_{*}$ of the system (\ref{eq2}) with parameter values: $r= 1.37 > r_{c} = 1.320961640, k = 200, a_1 = 100, a_2= 100, b_1 = 0.5, b_2 = 0.5; c_1 = 1.8, c_2 = 1.8, \delta_1 = 0.82, \delta_2 = 0.62, e_1 = 0.8143, e_2 = 0.6250.$}}\label{f1}
\end{figure}

\begin{figure}[tbhp]
\begin{center}
\includegraphics[width=11cm]{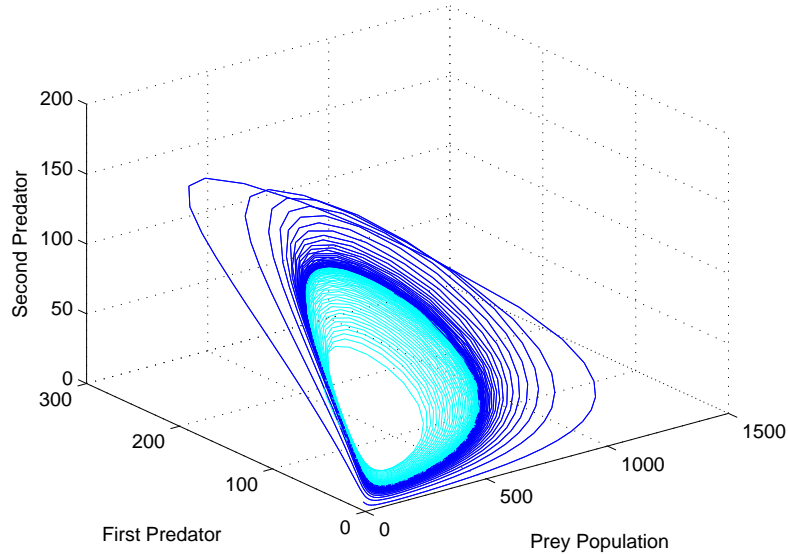}
\end{center}
\caption{Limit cycle behaviour of the dynamical system at $E_{*}$ with the same parameter values used for Figure \ref{f1}.}\label{f2}
\end{figure}

\begin{figure}
\centering
\includegraphics[width=14cm, height=12cm]{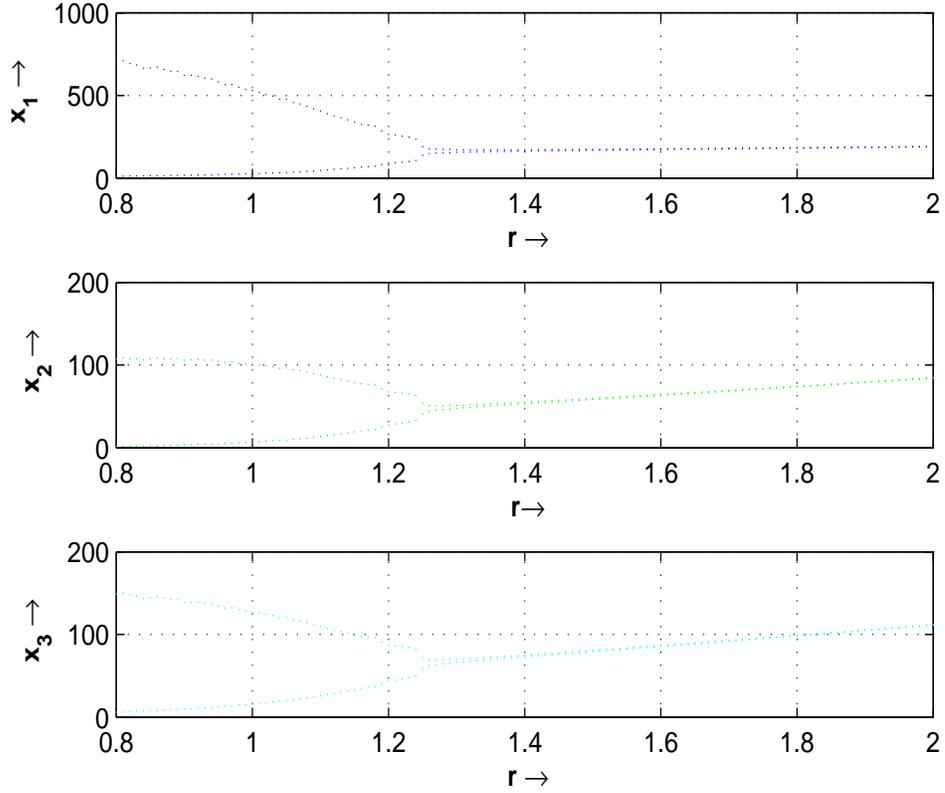}
 \caption{\textrm{\small Bifurcation diagram for all the populations with $r$ as the bifurcating parameter around the interior equilibrium $E_{*}$ of the system (\ref{eq2}).}}\label{f3}
\end{figure}

\begin{figure}
\centering
\includegraphics[width=11cm, height=6cm]{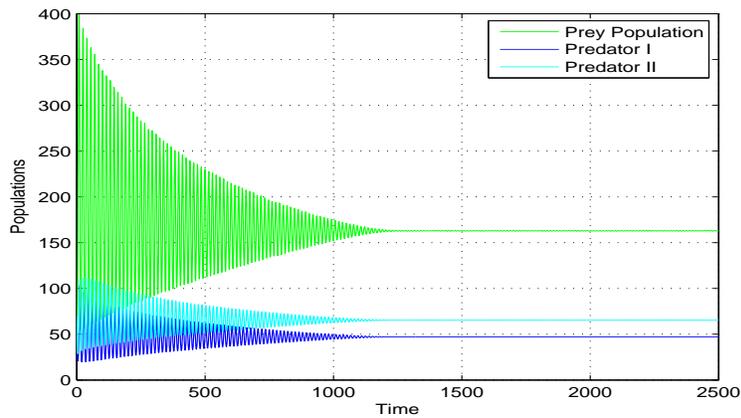}
 \caption{\textrm{\small 2D view of Local asymptotic stability of the system (\ref{eq2})  around the interior equilibrium $E_{*}$ of the system (\ref{eq2}) with parameter values: $r= 1.47 > r_{c} = 1.320961640, k = 200, a_1 = 100, a_2= 100, b_1 = 0.5, b_2 = 0.5; c_1 = 1.8, c_2 = 1.8, \delta_1 = 0.82, \delta_2 = 0.62, e_1 = 0.8143, e_2 = 0.6250.$}}\label{f4}
\end{figure}
\begin{figure}[tbhp]
\begin{center}
\includegraphics[width=11cm]{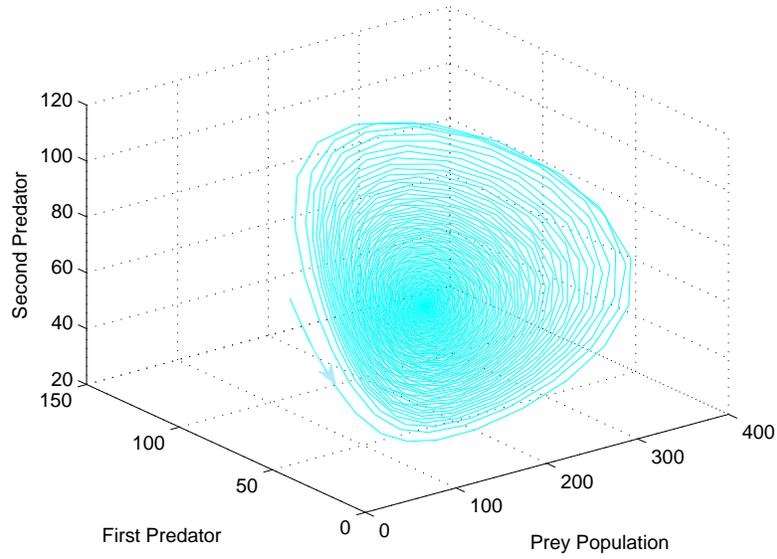}
\end{center}
\caption{3D view of local asymptotic stability of the dynamical system at $E_{*}$ with the same parameter values used for Figure \ref{f4}.}\label{f5}
\end{figure}
\begin{figure}[tbhp]
\begin{center}
\includegraphics[width=11cm]{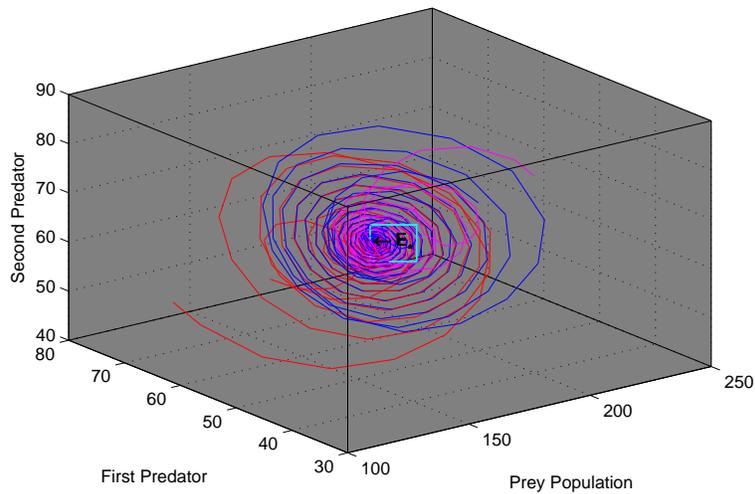}
\end{center}
\caption{Solution plots with different starting points converge to the interior equilibrium point $E_{*}=(169.1663564,55.36073780, 62.98120968),$ showing that the system (\ref{eq2}) is global asymptotic stable. Here the same set of parameter values is used for Figure \ref{f4} except $r=1.29.$}\label{f6}
\end{figure}

\begin{figure}[tbhp]
\begin{center}
\includegraphics[width=11cm]{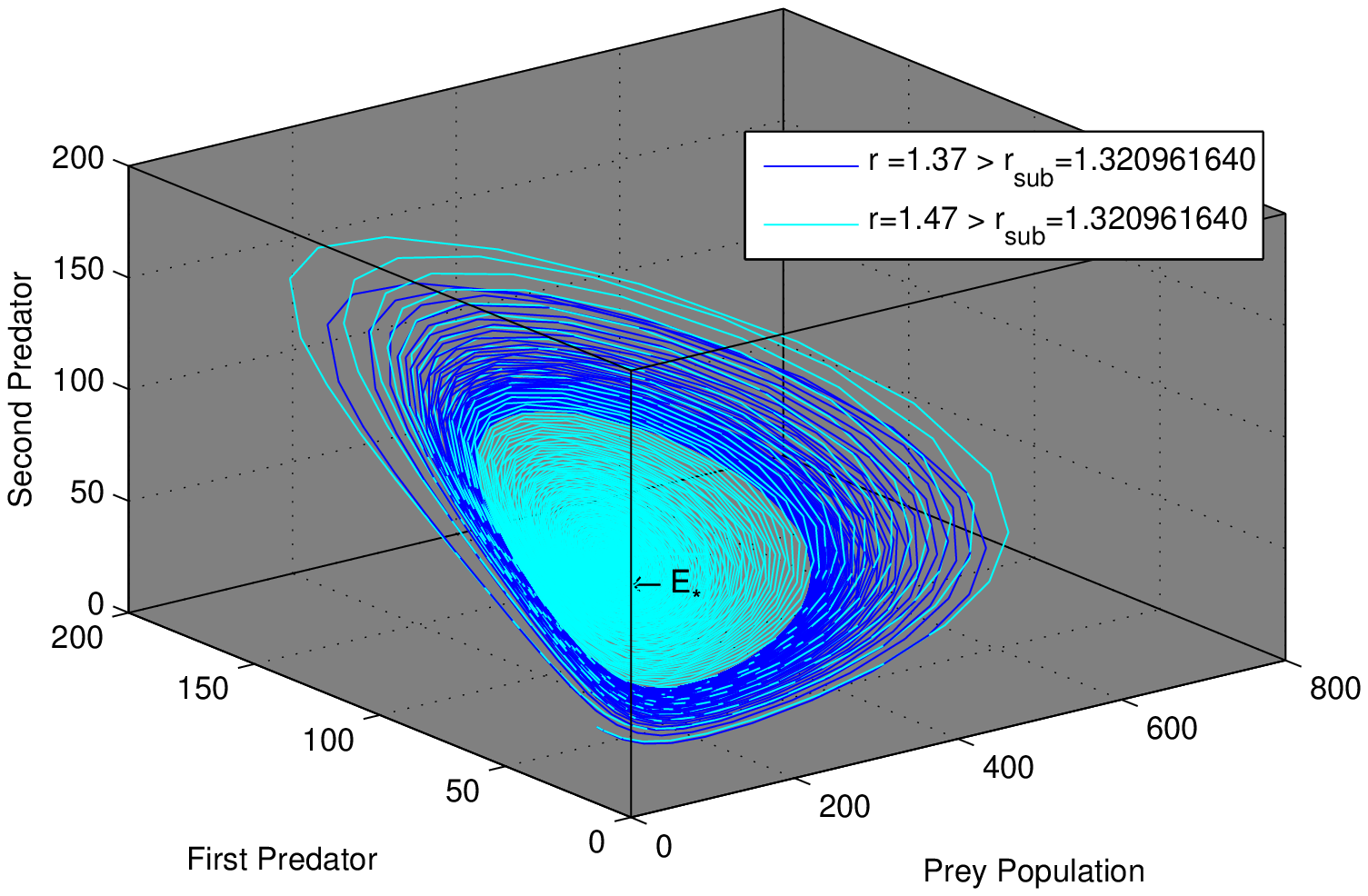}
\end{center}
\caption{Solution plots showing that the system (\ref{eq2}) experiences subcritical Hopf bifurcation for $r>r_{\text{sub}}.$ Here the set of parameter values used is mention in the last row of Table 2.}\label{f7}
\end{figure}

\section{Concluding remarks}\label{Discussion}
The problem describes by the system (\ref{eq2}) is well posed that $x_1,$ $x_{2}$ and $x_{3}$ axes are invariant under the flow of the system. So far our knowledge goes this is the first attempt to study an ecological system with semilinear/bilinear growth of the prey population. Generally, researcher only studied biological model systems with logistic/linear growth of prey population. Here is the novelty of our study. One of the important observations is that the prey population becomes unbounded in absence of its admissible predator in long run of time. But in the presence of predator species the prey population can be made bounded under suitable combination of system parameters and as a consequence it is shown that the total environmental population under consideration is bounded above (cf. Subsection \ref{bounded}). Therefore, any solution starting in the interior of the first octant never leaves it. This mathematical fact is consistent with the biological interpretation of the system. Due to the inclusion of semilinear/bilinear growth of the prey population, the axial equilibrium point is driven away by the system, which is rarely found in the modern research work on Mathematical biology. Thus, the prey population alone can not survive in stable condition without their admissible predator populations. It is found that only the mutual interference between the predators, which are parameterized by $b_{1}$ and $b_{2}$ can alone able to stabilize the prey-predator interactions even when a semilinear/bilinear intrinsic growth rate of prey population is considered in the proposed mathematical model. Whereas these parameters have much contribution in stabilizing prey-predator interactions when only linear intrinsic growth rate is considered in some mathematical models (cf. Dimitrov and  Kojouharov \cite{Dimitrov}).  It is observed in the study of this model system that there exist a balance between the predator's need for food and its saturation level and in this case is likely to be expect a periodic behaviour in long run. This behaviour is neutrally stable but relatively unstable. A small change in the parameters (caused by environmental changes for instances) forces the system to stabilize around the interior equilibrium or to oscillate indefinitely around interior equilibrium (by going away from it, which causes collapse of the system or breaks the coexistence of the population). Representative numerical simulations of this case are shown in Figures: \ref{f1}-\ref{f3}, which support our analytical findings (cf. Theorems \ref{theorem1} and \ref{theorem2}). We have also established the sufficient conditions for the global stability of the coexistence equilibrium (cf. Figures: \ref{f5}-\ref{f6}).

{\bf Acknowledgement:} Authors are thankful to the Department of Mathematics, Aliah University for providing opportunities to perform the present work.  Dr. S. Sarwardi is thankful to his Ph.D. supervisor Prof. Prashanta Kumar Mandal, Department of Mathematics, Visva-Bharati (a Central University) for his generous help while preparing this manuscript.


\begin{thebibliography}{99}
\addtolength{\leftmargin}{0.2in}
\setlength{\itemindent}{-0.25in}
\bibitem{Carr} Carr, J.: Applications of centre manifold theory. Springer-Verlag, {New York} (1981)

 \bibitem{May81} Anderson, R.M., May, R.M.: The population dynamics of microparasites and their invertebrates hosts. {Proc. R. Soc. London.}  {\bf 291}, 451--463 ( 1981)

 \bibitem{Beretta98} Beretta, E., Kuang, Y.: Global analysis in some delayed ratio-dependent predator-prey systems. {Nonl. Anal.}  {\bf 32}, 381--408 (1998)
 \bibitem{F90} Freedman, H.I.: A model of predator -prey dynamics modified by the action of parasite. {Math. Biosci.} {\bf 99}, 143--155 (1990)


\bibitem{HF89} Hadeler, K.P., Freedman, H.I.: Predator-prey populations with parasitic infection. {J. Math. Biol.}
{\bf 27}, 609--631 (1989)

\bibitem{HWM04} Hethcote, H.W., Wang, W., Ma, Z.: A predator prey model with infected prey. {Theor. Popul. Biol.}  {\bf 66}, 259--268 (2004)

\bibitem{Ma98} Ma, W.B., Takeuchi, Y.: Stability analysis on predator-prey system with distributed delays. {J. Comput. Appl. Math.}  {\bf 88}, 79--94 (1998)

\bibitem{V95} Venturino, E.: Epidemics in predator-prey models: disease in prey, in mathematical population dynamics. Analysis of
heterogeneity 1(Eds. Arino O, Axelrod D, Kimmel M, Langlais M.). 381--393 (1995)

\bibitem{XiaCh01} Xiao, Y., Chen, L.: Modeling and analysis of a predator-prey model with disease in prey. { Math. Biosci.} {\bf 171}, 59--82 (2001)

\bibitem{Cantrell}
 Cantrell, R.S., Cosner, C.: On the dynamics of predator-prey models with the Beddington-
DeAngelis functional response. {J. Math. Anal. Appl.}  {bf 257}, 206--222 (2001)

\bibitem{Cosner99} Cosner, C., Angelis, D.L., Ault, J.S., Olson, D.B.: Effects of spatial grouping on functional response of predators. {Theor. Popul. Biol.}  {\bf 56}, 65--75   (1999)
\bibitem{Cui} Cui, J., Takeuchi, Y.: Permanence, extinction and periodic solution of predator-prey system with
Beddington-DeAngelis functional response. {J. Math. Anal. Appl.} {\bf 317}, 464--474 (2006)



\bibitem{Huo}  Huo, H.F., Li. W.T., Nieto, J.J.: Periodic solutions of delayed predator-prey model with the
Beddington-DeAngelis functional response. {Chaos. Soli. Frac.}  {\bf 33}, 505--512 (2007)

\bibitem{Hwang} Hwang, T.W.: Global analysis of the predator-prey system with Beddington-DeAngelis functional
response. {J. Math. Anal. Appl.} {\bf 281}, 395--401 (2003)


\bibitem{Curds}  Curds, C.R.,  Cockburn, A.: Studies on the growth and feeding of Tetrahymena
pyriformis in axenic and monoxenic culture. {J. Gen. Microbiol.} {\bf 54}, 343--358 (1968)

\bibitem{Hassell} Hassell, M.P., Varley, G.C.: New inductive population model for insect parasites
and its bearing on biological control. {Nature.} {\bf 223}, 1133--1137 (1969)

\bibitem{Salt}  Salt, G.W.: Predator and prey densities as controls of the rate of capture by the
predator Didinium nasutum. {Ecology.} {\bf 55}, 434-–439 (1974)

\bibitem{Ginzburg}   Arditi, R., Ginzburg, L.R.: Coupling in predator–prey dynamics: ratiodependence.
{J. Theor. Biol.} {\bf 139}, 311--326 (1989)

\bibitem{Beddington75} Beddington, J.R.: Mutual interference between parasites or predators and its effect on searching
efficiency. {J. Anim. Ecol.} {\bf 44}, 331--340 (1975)

\bibitem{DeAngeli75} DeAngelis, R.A., Goldstein, R.A., Neill R.: A model of trophic interaction. {Ecology.}  {\bf56}, 881--892 (1975)

\bibitem{GaHa} Gard, T.C., Hallam, T.G.:  Persistence in Food web-1, Lotka-Volterra food chains.
 {Bull. Math. Biol.} {\bf41}, 877--891 (1979)

\bibitem{birkhoff1959ordinary} Birkhoff, G., Rota, G.C.: Ordinary Differential Equations. Ginn Boston (1982)


\bibitem{HV} Haque, M., Venturino, E.: Increase of the prey may decrease
the healthy predator population in presence of a disease in the predator. {Hermis} {\bf 7}, 39--60 (2006)

\bibitem{sarwardijbp13} Sarwardi, S., Mandal, P.K., Ray, S.: Dynamical behaviour of a two-predator model with prey refuge. {J. Biol. Phys.}  {\bf 39}, 701--722 (2013)

\bibitem{Sarwardirefuge} Sarwardi, S., Mandal, P.K., Ray, S.:  Analysis of a competitive prey-predator system with a prey refuge. {Biosystems}  {\bf 110}, 133--148 (2012)

  \bibitem{HV06} Haque, M., Venturino, E.: The role of transmissible diseases in the Holling–Tanner
predator–prey model.  {Theor. Popul. Biol.} {\bf 70}, 273--288 (2006)

\bibitem{wiggins} Wiggins, S.: 2003. Introduction to Applied Nonlinear Dynamical Systems and Chaos. Second Edition. Springer: New York.



    \bibitem{Kar} Kar, T.K., Gorai, A., Jana, S.:  Dynamics of pest and its predator model with disease in the pest and optimal use of pesticide. {J. Theor. Biol.} {\bf 310}, 187--198 (2012)


\bibitem{sarwardiglobal}  Sarwardi, S.,  Haque, M., Venturino, E.: Global stability and persistence in LG-Holling type-II diseased predators ecosystems. {J. Biol. Phys.} {\bf 37}, 91--106 (2010)


\bibitem{Hale} Hale, J.K.: Ordinary Differential Equations. Krieger Publisher Company, Malabar (1989)

\bibitem{Dimitrov} Dimitrov, D.T., Kojouharov, H.V.: Complete mathematical analysis of predator-prey models with linear prey growth and Beddington-DeAngelis functional response. {Appl. Math. Comp.} {\bf 162}, 523--538 (2005)
    \end{thebibliography}
\end{document}